\documentclass[12pt]{amsart}

\usepackage{cases}
\usepackage[english]{babel}
\usepackage{times,bm,amsfonts,amsmath,amssymb,dsfont} 
\usepackage{graphicx} 
\usepackage{fig4tex}
\usepackage[babel=true]{csquotes}
\usepackage{mathabx}
\usepackage{pdfsync}

\makeatletter
\def\section{\@startsection{section}{1}\z@{.9\linespacing\@plus\linespacing}%
  {.7\linespacing} {\fontsize{13}{15}\selectfont\scshape\centering}}
\def\paragraph{\@startsection{paragraph}{4}%
  \z@{0.3em}{-.5em}%
  {$\bullet$ \ \normalfont\itshape}}
\makeatother

\newtheorem{theo}{Theorem}[section]
\newtheorem{prop}[theo]{Proposition}
\newtheorem{lem}[theo]{Lemma}
\newtheorem{cor}[theo]{Corollary}

\theoremstyle{definition}

\theoremstyle{remark}
\newtheorem{rem}[theo]{Remark}
\newcommand\got[1]{{\bm{\mathfrak{#1}}}}
\makeatletter

\@addtoreset{equation}{section}
\makeatother

\usepackage{color}

\definecolor{gr}{rgb}   {0.,   0.69,   0.23 }
\definecolor{bl}{rgb}   {0.,   0.5,   1. }
\definecolor{mg}{rgb}   {0.85,  0.,    0.85}
\definecolor{yl}{rgb}   {0.8,  0.7,   0.}

\definecolor{webred}{rgb}{0.75,0,0}
\definecolor{webgreen}{rgb}{0,0.75,0}
\usepackage[citecolor=webgreen,colorlinks=true,linkcolor=webred]{hyperref}

\renewcommand{\d}{\, {\rm d}}

\newcommand{\R}{\mathbb{R}}

\newcommand{\Rp}{\R_{+}}
\newcommand{\Hplow}{\rm{Hp}^{\rm low}}
\newcommand{\Hslow}{\rm{Hs}^{\rm low}}
\newcommand{\Hpup}{\rm{Hp}^{\rm up}}

\newcommand{\spectre}{\lambda}

\newcommand{\spec}{\got{S}}
\newcommand{\specess}{\got{S}_{\mathrm{ess}}}

\newcommand{\dom}{\operatorname{Dom}}
\newcommand{\supp}{\operatorname{Supp}}

\renewcommand{\Im}{\operatorname{Im}}

\newcommand{\curl}{\operatorname{curl}}
\newcommand{\dist}{\operatorname{dist}}

\newcommand{\sinc}{\operatorname{sinc}}

\newcommand{\Add}{\underline{{\bf A}}}
\newcommand{\B}{\underline{{\bf B}}}

\newcommand{\dg}{\got{h}^{ \rm N}}

\newcommand{\mudg}{\mu^{\rm N}}

\newcommand{\Secteur}{\mathcal{S}}

\newcommand\Wedge{\mathcal{W}}
\newcommand{\sd}{\lambda}
\newcommand{\sse}{\underline{s}}
\newcommand{\ssess}{\underline{s}_{\, \rm ess}}
\newcommand{\FQ}{\mathcal{Q}_{\tau}}
\newcommand{\FQpol}{\mathcal{Q}^{\rm pol}}

\newcommand{\Vpol}{V^{\rm pol}}

\title[Magnetic Schr\"odinger operator with tangent magnetic field]{\large The Schr\"odinger operator on an infinite wedge with a tangent magnetic field. }
\author{Nicolas Popoff}
\address{Laboratoire IRMAR, UMR 6625 du CNRS, Campus de Beaulieu, 35042 Rennes cedex, France
\\
 {\it E-mail adress:} nicolas.popoff@univ-rennes1.fr}

\date{\today}

\begin{document}
\begin{abstract}
We study a model Schr\"odinger operator with constant magnetic field on an infinite wedge with Neumann boundary condition. The magnetic field is assumed to be tangent to a face. We compare the bottom of the spectrum to the model spectral quantities coming from the regular case. We are particularly motivated by the influence of the magnetic field and the opening angle of the wedge on the spectrum of the model operator and we exhibit cases where the bottom of the spectrum is smaller than in the regular case. Numerical computations enlighten the theoretical approach.
\end{abstract} 
\maketitle

\section{Introduction}

\subsection{The magnetic Laplacian on model domains}
\paragraph{Motivation}
Let $(-ih\nabla-{\bf A})^2$ be the Schr\"odinger magnetic operator (also called the magnetic Laplacian) on an open simply connected subset $\Omega$ of $\R^3$.
 The magnetic potential ${\bf A}:\R^3\mapsto \R^3 $ satisfies $\curl {\bf A}={\bf B}$ where ${\bf B}$ is the magnetic field and $h$ is a semi-classical parameter. For a reasonable domain $\Omega$, the Neumann realization of $(-ih\nabla-{\bf A})^2$ is an essentially self-adjoint operator with compact resolvent. The motivation for the study of this operator comes from the theory of superconductivity, indeed the linearization of the Ginzburg-Landau functional brings the study of the Neumann magnetic Laplacian (see \cite{GiPh99}). For a magnetic field of strong intensity, the superconductivity phenomenon is destroyed. We denote by $\spectre({\bf B};\Omega,h)$ the first eigenvalue of $(-ih\nabla-{\bf A})^2$. The behavior of the critical value of the magnetic field for which the superconductivity disappears is linked to $\spectre({\bf B};\Omega,h)$ when $h$ goes to 0 (see \cite[Proposition 1.9]{FoHe3} for example). 
  
  A common interest is to understand the influence of the combined geometries of the domain $\Omega$ and the magnetic field ${\bf B}$ on the asymptotics of $\spectre({\bf B};\Omega,h)$ in the semi-classical limit $h\to0$.

\paragraph{Link between the semi-classical problem and model operators}
 In order to find the main term of the asymptotics of $\spectre({\bf B};\Omega,h)$, we are led to study the magnetic Laplacian without semi-classical parameter ($h=1$) on unbounded ``model" domains invariant by dilatation with a constant magnetic field. More precisely to each point $x\in \overline{\Omega}$ we associate its tangent cone $\Pi_{x}$
and we denote by
$$ P_{{\bf A}_{x}, \, \Pi_{x}}=(-i\nabla-{\bf A}_{x})^2$$
the Neumann realization of the magnetic Laplacian on the model domain $\Pi_{x}$ where ${\bf A}_{x}$ satisfies $\curl {\bf A}_{x}={\bf B}_{x}$ and where ${\bf B}_{x}$ is the constant vector field equal to ${\bf B}(x)$. We denote by
\begin{equation}
\label{D:spectremodel}
  \spectre({\bf B}_{x};\Pi_{x}) \ \ \mbox{ the bottom of the spectrum of} \ \ P_{{\bf A}_{x}, \, \Pi_{x}} \ .
\end{equation}
When the domain belongs to a suitable class of corner domains (see \cite[Chapter 1]{Dau88} for example) and if the magnetic field is regular and does not vanish, one should expect that $\spectre({\bf B};\Omega,h)$ behaves like $h \inf_{x\in\overline{\Omega}}\spectre({\bf B}_{x};\Pi_{x})$ when $h\to0$ \footnote{All the asymptotics known for particular domains have this structure. A work with M. Dauge and V. Bonnaillie-No\"el is in progress to get the behavior of $\spectre({\bf B};\Omega,h)$ at first order for general domains $\Omega$.}. To a constant magnetic field we can associate a linear potential and due to a scaling we have $\spectre({\bf B}_{x};\Pi_{x})=\|{\bf B}_{x}\|\spectre\left(\frac{{\bf B}_{x}}{\|{\bf B}_{x} \|};\Pi_{x}\right)$. Therefore when we will deal with the magnetic Laplacian on model domains, we will always suppose that the magnetic field is constant an unitary.

\paragraph{Regular case}
 When $\Omega$ is a 3D-domain with regular boundary, we only need to study the magnetic Laplacian on a space and on half-spaces for different orientations of the magnetic field. The bottom of the spectrum of the associated operators is minimal when $\Pi$ is a half-space and ${\bf B}$ is tangent to the boundary (see \cite{LuPan00} and \cite{HeMo02}). In that case we have $\spectre({\bf B};\Pi)=\Theta_{0}\approx 0.59$ (see \cite{SaGe63} for the first work on $\Theta_{0}$ or Subsection \ref{SS:demiplan} for more details and references). When ${\bf B}$ is constant and $\Omega\subset\R^3$ is regular, the following asymptotics is proved 
  in \cite{LuPan00} (see also \cite{HeMo04} for more terms): 
 \begin{equation}
 \spectre(\Omega;{\bf B},h)\underset{h\to0}{\sim}\Theta_{0}h \ 
 \end{equation}
 
 \paragraph{Singular cases known}
 When $\Omega$ has an edge, it is necessary to introduce a new model operator: the magnetic Laplacian on a infinite wedge. We denote by $\alpha$ the opening angle of the wedge. In \cite{Pan02}, Pan has studied the case of a wedge whose opening angle is $\frac{\pi}{2}$ and has applied its results to study the first eigenvalue of the magnetic Laplacian on a cuboid in the semi-classical limit. He proved that there exist configurations where the bottom of the spectrum of the magnetic Laplacian on a quarter space is smaller than the spectral quantity $\Theta_{0}$ coming from the regular case. Using the Neumann boundary condition and symmetrization, he compared the operators to the model operator on a half-plane. When the opening angle is different from $\frac{\pi}{2}$, we can not use this method anymore. 
 
 Another case already studied is the one of a magnetic field tangent to the axis of the wedge. The operator reduced to a 2D operator on a sector whose spectrum is studied in \cite{Ja01} for the special case $\alpha=\frac{\pi}{2}$ and in \cite{Bon06} for wedges of opening $\alpha\in (0,\pi)$. One of the main result is that for $\alpha\in(0,\frac{\pi}{2}]$, the bottom of the spectrum of this model operator is below $\Theta_{0}$.
 
  In \cite{PopRay12}, the authors deal with the case where $\Omega$ is a lens with a curved edge. The model operator involved is the magnetic Laplacian on an infinite wedge with a magnetic field normal to the plane of symmetry of the wedge. The results from \cite[Chapter 6]{Popoff} show that in that case the bottom of the spectrum of the model operator is always larger than $\Theta_{0}$ and is decreasing with the opening angle of the wedge.
 
   In this article we study the bottom of the spectrum of the magnetic Laplacian on infinite convex wedges in the case where the magnetic field is tangent to a face of the wedge. We compare the bottom of the spectrum to the model spectral quantity $\Theta_{0}$ and we characterize the spectrum of the 2D operator family associated to the magnetic Laplacian on the wedge. We are particularly interested in the influence of the magnetic field orientation and the opening angle of the wedge. Some of our results recover what was done in \cite{Pan02} for the quarter space and give a new approach using the tools of the spectral theory.
  
\subsection{The operator on a wedge}
Let $(x_{1},x_{2},x_{3})$ be the cartesian coordinates of $\R^3$. The infinite sector of opening $\alpha\in (0,\pi)$ is denoted by 
$$\Secteur_{\alpha}:=\{ (x_{1},x_{2})\in \R^2, \, |x_{2}| \leq x_{1}\tan\tfrac{\alpha}{2} \} \  $$
and the  infinite wedge of opening $\alpha$ is 
$$\Wedge_{\alpha}:=\Secteur_{\alpha}\times \R \ . $$
The magnetic field ${\bf B}=(b_{1},b_{2},b_{3})$ is constant and unitary and we denote by $\B:=(b_{1},b_{2})$ its projection on $\R^2$. The spherical coordinates are denoted by $(\gamma,\theta)$ and satisfied $\cos\gamma={\bf B}\cdot(0,0,1)$ and $\cos\theta=\B\cdot(0,1)$. We will assume that the magnetic field ${\bf B}$ is tangent to a face of the edge (see figure \ref{F:diedre}). Due to symmetry we will restrict our study to the case where $\gamma\in [0,\frac{\pi}{2}]$ and $\theta=\frac{\pi-\alpha}{2}$, and therefore the magnetic field writes
\begin{equation}
\label{D:Bcoordpolar}
{\bf B}=(\sin\gamma \cos\tfrac{\alpha}{2},\sin\gamma \sin\tfrac{\alpha}{2},\cos\gamma) \ .
\end{equation}
\begin{figure}[h!]
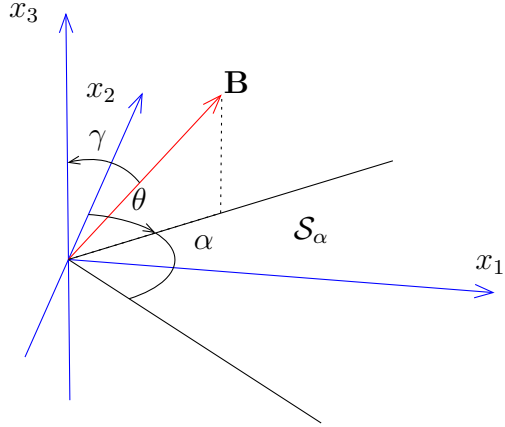

\begin{center}
         \figinit{pt,realistic}
\figset projection(psi=-80)
\def\dist{160}
\def\Ralpha{40}
\figpt 1:(0,0,0)
\figpttraC 2:=1/\dist,0,0/
\figpttraC 3 :=1/0,\dist,0/
\figpttraC 4:=1/0,0,100/
\figpt 5:(0,0,-60)
\figpttraC 6 :=1/110,110,0/
\figpttraC 10 :=1/110,-110,0/
\figgetdist \unite [1,6]
\figvectN 7 [1,4,6]
\figpttra 8 :$x_{2}$= 1 / \unite, 7/
\figpt 12:(0,-80,0)
\figpt 13:(50,50,50)
\figptorthoprojplane 14 :=13/1,4/ 
\figptcirc 9 :$M$:1,6,10;\Ralpha(50)
\figptcirc 15 :$M$:1,3,14;\Ralpha(20)
\figptcirc 16 :$M$:1,13,4;\Ralpha(32)

\psbeginfig{}
\psline[1,6]
\psline[1,10]
\psset(color=\Bluergb)
\psarrow[5,4]
\psarrow[1,2]
\psarrow[12,3]
\psset(color=\Redrgb)
\psarrow[1,13]
\psset(color=\Blackrgb)
\psset arrowhead(length=5)
\psarrowcircP 1;40 [3,14,14]
\psarrowcircP 1;40 [13,4,4]
\psarccircP 1;\Ralpha  [10,6,6]
\psset(dash=5)\psline[1,14,13]
\psendfig
\figvisu{\figBoxA}{}{
\figwriten 2:$x_{1}$ (7)
\figwritew 4:$x_{3}$ (7)
\figwritew 3:$x_{2}$ (7)
\figwritene 9:$\alpha$ (12)
\figwritene 13:${\bf B}$ (1)
\figwritene 15:$\theta$ (5)
\figwriten 16:$\gamma$ (4)
\figwritesw 6:$\Secteur_{\alpha}$ (30)
}
\par\centerline{\box\figBoxA}
\caption{The infinite wedge $\Wedge_{\alpha}$ of opening $\alpha$ and the magnetic field ${\bf B}$ of spherical coordinate $(\gamma,\theta)$.}
\label{F:diedre}
\end{center}
\end{figure}
\\
We assume that the magnetic potential ${\bf A}=(a_{1},a_{2},a_{3})$ satisfies $\curl {\bf A}={\bf B}$ 
and the magnetic Schr\"odinger operator writes: 
$$P_{{\bf A}, \, \Wedge_{\alpha}}=(D_{x_{1}}-a_{1})^2+(D_{x_{2}}-a_{2})^2+(D_{x_{3}}-a_{3})^2 \ .$$
with $D_{x_{j}}=-i\partial_{x_{j}}$. Due to gauge invariance, the spectrum of $P_{{\bf A}, \, \Wedge_{\alpha}}$ does not depend on the choice of ${\bf A}$ as soon as it satisfies $\curl {\bf A}={\bf B}$ and we will denote by ``choice of gauge" the choice of a magnetic potential that satisfies $\curl {\bf A}={\bf B}$. According to \eqref{D:spectremodel} we note:  
$$\sd({\bf B};\Wedge_{\alpha}):=\inf \spec(P_{{\bf A}, \, \Wedge_{\alpha}}) \ , $$
where we denote by $\spec(P)$ the spectrum of an operator $P$. We also denote by $\specess(P)$ the essential spectrum an operator $P$. Due to the invariance by translation in the $x_{3}$-variable, the spectrum of $P_{{\bf A}, \, \Wedge_{\alpha}}$ is absolutely continuous and we have $\spec(P_{{\bf A}, \, \Wedge_{\alpha}})=\specess(P_{{\bf A}, \, \Wedge_{\alpha}})$.

\paragraph{Reduction to a parameter family of operators on the sector}
We take a magnetic potential of the form ${\bf A}(x_{1},x_{2},x_{3})=(\Add(x_{1},x_{2}),x_{2}b_{1}-x_{1}b_{2})$ where the 2D-magnetic potential $\Add$ satisfies $\curl \Add=b_{3}$. An example for the choice of $\Add$ is the ``Landau" potential $\Add^{\rm L}(x_{1},x_{2})=(-x_{2}b_{3},0)$ and the associated operator writes 
$$P_{{\bf A}^{\rm L}, \, \Wedge_{\alpha}}=(D_{x_{1}}+x_{2}b_{3})^2+D_{x_{2}}^2+(D_{x_{3}}-x_{2}b_{1}+x_{1}b_{2})^2 \ .$$
We introduce the reduced electric potential on the sector: 
$$  V_{\B, \, \tau}(x_{1},x_{2}):=(x_{1}b_{2}-x_{2}b_{1}-\tau)^2 \ ,$$
where the Fourier parameter $\tau$ lies in $\R$. Performing a Fourier transform in the $x_{3}$ variable, we get the following direct integral decomposition (see \cite{ReSi78}): 
\begin{equation}
\label{E:decompintdir1}
P_{{\bf A},\, \Wedge_{\alpha}} =\int_{\tau\in\R}^{\bigoplus} P_{\Add,\, \Secteur_{\alpha}}+ V_{\B, \, \tau} \d \tau \ .
\end{equation}
where $ P_{\Add,\, \Secteur_{\alpha}}+ V_{\B, \, \tau}$ is the Neumann realization of $(-i\nabla-\Add)^2+V_{\B, \, \tau}$ on the sector $\Secteur_{\alpha}$.
Let us define
 $$\sse({\bf B};\alpha,\tau):=\inf\spec(P_{\Add,\, \Secteur_{\alpha}}+ V_{\B, \, \tau})$$ and 
$\FQ $ the quadratic form associated to $P_{\Add,\, \Secteur_{\alpha}}+ V_{\B, \, \tau}$. It is elementary that the form domain of $P_{\Add,\, \Secteur_{\alpha}}+ V_{\B, \, \tau}$ is 
$$\dom(\FQ)= \{u\in L^2(\Secteur_{\alpha}), \, (-i\nabla-\Add)u\in L^2(\Secteur_{\alpha}), \, |x_{1}b_{2}-x_{2}b_{1}|u\in L^2(\Secteur_{\alpha}) \} $$
and for $u\in \dom(\FQ)$ the expression of the quadratic form is 
$$ \FQ(u):=\int_{\Secteur_{\alpha}}|(-i\nabla-\Add)u|^2+V_{\B, \, \tau}|u|^2 \d x_{1} \d x_{2} \ . $$
Since the form domain does not depend on $\tau$, from Kato's perturbation theory (see \cite{kato}) the function $\tau\mapsto\sse({\bf B};\alpha,\tau)$ is continuous on $\R$. Thanks to \eqref{E:decompintdir1} we have the fundamental relation, sometimes called the F-principle (see \cite{LuPan00_2}): 
\begin{equation}
\label{E:relationfond}
\sd({\bf B};\Wedge_{\alpha})=\inf_{\tau\in\R} \sse({\bf B};\alpha,\tau)
\end{equation}
Therefore we are reduced to study the spectrum of a 2D-family of Schr\"odinger operators. 
\paragraph{Invariance principles}
We recall the action of isometry on the 2D-magnetic Laplacian: 
\begin{itemize}
\item Translation: let $\Omega\in \R^2$ and ${\bf t}\in \R^2$. Let $\Omega_{{\bf t}}:=\Omega+{\bf t}$ be the domain deduced by translation. Let $\Add$ be a magnetic field such that $\curl \Add$ is a constant denoted by $B$. Then $P_{\Add, \, \Omega}$ and $P_{\Add, \, \Omega_{{\bf t}}}$ are unitary equivalent, moreover $u$ is an eigenfunction for $P_{\Add, \, \Omega}$ if and only if $x \mapsto e^{\frac{iB}{2}x\wedge {\bf t}}u(x-{\bf t})$ is an eigenfunction for $P_{\Add, \, \Omega_{{\bf t}}}$.
\item Rotation: let $\Omega\in \R^2$ and $R_{\omega}$ be the rotation of angle $\omega$. Let $\Omega_{\omega}:=R_{\omega}(\Omega)$ be the domain deduced by rotation. Then $P_{\Add, \, \Omega}$ and $P_{\Add, \, \Omega_{\omega}}$ are unitary equivalent, moreover $u$ is an eigenfunction for $P_{\Add, \, \Omega}$ if and only if $u\circ R_{\omega}^{-1}$ is an eigenfunction for $P_{\Add, \, \Omega_{\omega}}$.
\end{itemize}
\subsection{Problematic}
We study the spectral quantity $\sd({\bf B};\Wedge_{\alpha})$ and the associated ``band function" $\tau\mapsto\sse({\bf B};\alpha,\tau)$. We are particularly interested in the following questions: 
\begin{itemize}
\item Does the band function $\tau \mapsto \sse({\bf B};\alpha,\tau)$ reach its infimum?
\item If it does, is this infimum a discrete eigenvalue for the operator $P_{\Add,\, \Secteur_{\alpha}}+ V_{\B, \, \tau}$?
\item Is it possible to compare $\sd({\bf B};\Wedge{\alpha})$ and $\Theta_{0}$?
\end{itemize}
In \cite{Pan02}, these questions are partially answered for the special case $\alpha=\frac{\pi}{2}$. It is proved that the band function $\tau\mapsto \sse({\bf B};\frac{\pi}{2},\tau)$ always reaches its infimum and that $\sd({\bf B};\Wedge_{\frac{\pi}{2}})<\Theta_{0}$ when the magnetic field is tangent to a face, except when it is normal to the axis of the wedge, and it this case $\sd({\bf B};\Wedge_{\frac{\pi}{2}})=\Theta_{0}$. As said before the proofs are specific to the case $\alpha=\frac{\pi}{2}$ and the general case cannot be deduced using the same arguments.
\subsection{Organization of the paper}
In Section \ref{S:modeloperators} we recall results about model operators and we introduce auxiliary operators linked to the behavior of the operator on the wedge at infinity. In Section \ref{S:specess} we determine the bottom of the essential spectrum of the operator $P_{\Add,\, \Secteur_{\alpha}}+ V_{\B, \, \tau}$ on the sector. In Section \ref{S:Limit} we compute the limit of $\sse({\bf B};\alpha,\tau)$ when $\tau\to-\infty$ and $\tau\to+\infty$. We provide an explicit expression for these limits using the spectral model quantity coming from the problem on the half-plane. In Section \ref{S:ub} we construct quasi-modes for the operator on the sector and we deduce a rough upper bound for $\sd({\bf B};\Wedge_{\alpha})$. In Section \ref{S:champnormal} we study the special case where the magnetic field is tangent to a face and normal to the axis of the wedge. In Section \ref{S:numericalsimu} we present several numerical computations of the first eigenpair of $P_{\Add,\, \Secteur_{\alpha}}+ V_{\B, \, \tau}$. 

\section{Model and auxiliary operators}
\label{S:modeloperators}
In this section we recall results about the bottom of the spectrum of the magnetic Laplacian in model domains. 
\subsection{The half-space}
\label{SS:demiplan}
Let $\R^3_{+}:=\{(s,t,z)\in \R^3, t>0\}$ be the model half-space. We assume that the constant unitary magnetic field ${\bf B}_{\theta}$ makes an angle $\theta$ with the boundary of $\R^3_{+}$. Thanks to symmetries, we only need to study $\theta\in [0,\frac{\pi}{2}]$.   
\paragraph{Tangent case: the de Gennes operator}
 Here we assume that the magnetic field ${\bf B}_{0}$ is tangent to the boundary, then in a suitable gauge, the magnetic operator writes 
$$P_{{\bf A}_{0}, \, \R^3_{+}}=(D_{s}+t)^2+D_{t}^2+D_{z}^2 \ . $$
Using a Fourier transform in the variables $(s,z)$ we have 
\begin{equation}
\label{E:directintcastangent}
P_{{\bf A}_{0}, \, \R^3_{+}}=\int_{(\tau,k)\in\R^2}^{\bigoplus} \dg_{\tau}+k^2 \d \tau \d k\  
\end{equation}
where the de Gennes operator $\dg_{\tau}$ is defined as the following 1D-operator:
$$ \dg_{\tau}:=D_{t}^2+(t-\tau)^2, \quad t>0 $$
on the domain
\begin{equation}
\label{D:B2Neu}
B^{2}_{\rm Neu}(\R_{+}):=\{u\in H^2(\R_{+}), \, t^2u \in L^2(\R_{+}), \, u'(0)=0 \} \ .
\end{equation}
 This operator has compact resolvent and we define 
\begin{equation}
\label{D:mudg}
\mudg_{1}(\tau):=\inf \spec (\dg_{\tau})
\end{equation}
 its first eigenvalue. We have (see \cite{HePer10} and also \cite{FouHel10}): 
$$ \lim_{\tau\to -\infty}\mudg_{1}(\tau)=+\infty \quad \mbox{ and } \quad \lim_{\tau\to+\infty}\mudg_{1}(\tau)=1 \ . $$
It is shown in \cite{Bol92} and \cite{DauHe93} that it exists $\xi_{0}>0$ such that the function $\tau \mapsto \mudg_{1}(\tau)$ is decreasing on $(-\infty,\xi_{0}]$ and increasing on $[\xi_{0},+\infty)$, therefore it has a unique minimum denoted by $\Theta_{0}$, in addition this minimum is non-degenerate and we have $\xi_{0}^2=\Theta_{0}$. Refined numerical computations coming from \cite{Bo08} provide the following approximation with an error inferior to $10^{-9}$:
 \begin{equation}
\label{valeur-numerique-theta0}
\Theta_{0}\simeq 0.590106125 \, \text{ and } \, \xi_{0}\simeq0.76818365314 \ .
\end{equation}
Due to \eqref{E:directintcastangent},  when $\curl {\bf A}$ is tangent to the boundary of $\R^3_{+}$ we have:
$$\spec(P_{{\bf A}, \, \R^3_{+}})=[\Theta_{0},+\infty) \ . $$
\paragraph{Non tangent case}
We now assume that the magnetic field makes an angle $\theta\in(0,\frac{\pi}{2}]$ with the boundary of $\R^3_{+}$. After using a rotation, we take ${\bf B}_{\theta}=(\cos\theta,\sin\theta,0)$ and we choose an associated magnetic potential by taking ${\bf A}_{\theta}(s,t,z)=(0,0,t\cos\theta-s\sin\theta)$. The magnetic Laplacian writes:
$$P_{{\bf A}_{\theta}, \, \R^3_{+}}=D_{s}^2+D_{t}^2+(D_{z}-t\cos\theta+s\sin\theta)^2 \ . $$
We introduce the bottom of its spectrum 
\begin{equation}
\label{D:sigma}
\sigma(\theta):=\inf\spec(P_{{\bf A}_{\theta}, \, \R^3_{+}}) \ . 
\end{equation}
This model spectral quantity has been widely studied (see  \cite{LuPan00_2}, \cite{LuPan00}, \cite{HeMo02} \cite{MoTr05} or more recently \cite{BoDauPopRay12}). Let us recall that the function $\theta\mapsto \sigma(\theta)$ is increasing from $(0,\frac{\pi}{2}]$ onto $(\Theta_{0},1]$ (see \cite{LuPan00}).
\subsection{The wedge with a magnetic field tangent to the edge}
\label{SS:opmodelewedgetangent}
 We deal with the case where the magnetic field ${\bf B}=(0,0,1)$ is tangent to the edge $\{x_{3}=0\}$. In that case the electric potential on the sector is $V_{\B, \, \tau}=\tau^2$ and we have 
 $$\sse({\bf B};\alpha,\tau)=\mu(\alpha)+\tau^2$$
 where $\mu(\alpha)$ is the bottom of the spectrum of $P_{\Add, \, \Secteur_{\alpha}}$ with $\curl \Add=1$. Therefore thanks to \eqref{E:relationfond} we get in that case: 
 $$\sd({\bf B};\alpha)=\mu(\alpha) \ .$$
 Let us gather results coming from \cite{Bon06} about the model operator $P_{\Add, \, \Secteur_{\alpha}}$:
 \begin{prop}
 Let $\Add$ be a 2D-magnetic potential such that $\curl \Add=1$, $P_{\Add, \, \Secteur_{\alpha}}$ the associated magnetic Laplacian on the sector $\Secteur_{\alpha}$ and $\mu(\alpha)=\inf \spec(P_{\Add, \, \Secteur_{\alpha}})$. Then we have: 
 \begin{enumerate}
\item $\specess(P_{\Add, \, \Secteur_{\alpha}})=[\Theta_{0},+\infty)$,
 \item\label{mulestheta0} $\forall \alpha\in (0,\frac{\pi}{2}]$, $\mu(\alpha)<\Theta_{0}$,
 \item Asymptotics for the small angle limit: $\mu(\alpha)\underset{\alpha\to0}{\sim}\frac{\alpha}{\sqrt{3}}$.
 \end{enumerate}
 \end{prop}
 Numerical simulations coming from \cite{BoDauMaVial07} show that \eqref{mulestheta0} seems to hold for all $\alpha\in(0,\pi)$. In addition $\alpha\mapsto \mu(\alpha)$ seems to be increasing for $\alpha\in(0,\pi)$. These two problems are still open.
 
 \subsection{Auxiliary operators}
 Let us define the half-planes $\Hpup:=\{x_{2}<x_{1}\tan\frac{\alpha}{2} \}$ and $\Hplow:=\{x_{2}>-x_{1}\tan\frac{\alpha}{2} \}$ such that $\Secteur_{\alpha}=\Hpup\cap \Hplow$. In this section we study the operators $(D_{x_{1}}-x_{2}b_{3})^2+D_{x_{2}}^2+(x_{1}b_{2}-x_{2}b_{1}-\tau)^2$ acting on $L^2(\Hpup)$ and $L^2(\Hplow)$ with Neumann boundary condition. We denote by $P_{\Add,\, \Hpup}+ V_{\B, \, \tau}$ and $P_{\Add,\, \Hplow}+ V_{\B, \, \tau}$ these two operators. They have been introduced in \cite[Section 5]{Pan02} where the author gives bounds for the bottom of their spectrum. In this section we give explicit formulae using the spectral model quantities $\mudg_{1}$ and $\sigma$ coming from the previous Subsection.
 \paragraph{Operators for the upper boundary}
\begin{lem}
\label{L:Opmodeledemiplanup}
Assume that the magnetic field writes ${\bf B}=(\sin\gamma \cos\tfrac{\alpha}{2},\sin\gamma \sin\tfrac{\alpha}{2},\cos\gamma)$. Then we have
\begin{equation}
\label{R:opmodelmudg}
\inf \spec\left(P_{\Add,\, \Hpup}+ V_{\B, \, \tau}\right)=\inf_{\xi_{2}\in\R}\left(\mudg_{1}(\xi_{2}\cos\gamma+\tau\sin\gamma)+(\xi_{2}\sin\gamma-\tau\cos\gamma)^2\right)
\end{equation}
\end{lem}
\begin{proof}For a suitable choice of gauge the expression of the operator is 
$$P_{\Add,\, \Hpup}+ V_{\B, \, \tau}=(D_{x_{1}})^2+(D_{x_{2}}-x_{1}\cos\gamma)^2+(x_{1}\sin\gamma\sin\tfrac{\alpha}{2}-x_{2}\sin\gamma\cos\tfrac{\alpha}{2}-\tau)^2$$
where $(\gamma,\frac{\pi-\alpha}{2})$ are the spherical coordinates of ${\bf B}$. Using a rotation of angle $\frac{\pi-\alpha}{2}$ and a change of gauge, the operator $P_{\Add,\, \Hpup}+ V_{\B, \, \tau}$ is unitary equivalent to the Neumann realization of
$$(D_{s}-t\cos\gamma)^2+D_{t}^2+(t\sin\gamma-\tau)^2 \ , \quad (s,t)\in \R^2_{+}  \ ,$$
where $\R^2_{+}:=\{(s,t)\in \R^2, t>0 \}$.  Making a partial Fourier transform in the $s$ variable, we get that 
$$(D_{s}-t\cos\gamma)^2+D_{t}^2+(t\sin\gamma-\tau)^2 =\int_{\xi_{2}\in\R}^{\bigoplus}D_{t}^2+(\xi_{2}-t\cos\gamma)^2+(t\sin\gamma-\tau)^2 \ \d \xi_{2} $$
where the operator $D_{t}^2+(\xi_{2}-t\cos\gamma)^2+(t\sin\gamma-\tau)^2$ acts on the functions of the variable $t$ belonging to $B^{2}_{\rm Neu}(\R_{+})$ (see \eqref{D:B2Neu}). Since we have for fixed $\xi_{2}\in\R$:
\begin{multline*}
\inf\spec \left(D_{t}^2+(\xi_{2}-t\cos\gamma)^2+(t\sin^2\gamma-\tau)^2\right)\\=\inf\spec \left(D_{t}^2+(t-\tau\sin\gamma-\xi_{2}\cos\gamma)^2\right)+(\xi_{2}\sin\gamma-\tau\cos\gamma)^2 \ ,
\end{multline*}
we get \eqref{R:opmodelmudg} using \eqref{D:mudg}.
\end{proof}

 \paragraph{Operators for the lower boundary}
\begin{lem}
\label{L:Opmodeledemiplanlow}
Assume that the magnetic field writes ${\bf B}=(\sin\gamma \cos\tfrac{\alpha}{2},\sin\gamma \sin\tfrac{\alpha}{2},\cos\gamma)$. Then the spectrum of $P_{\Add,\, \Hplow}+ V_{\B, \, \tau}$ does not depend of $\tau$ and we have
\begin{equation}
\label{R:opmodelsigma}
\forall \tau \in \R, \quad \inf \spec\left(P_{\Add,\, \Hplow}+ V_{\B, \, \tau}\right)=\sigma(\beta)
\end{equation}
with $\beta=\arcsin(\sin\alpha \sin\gamma)$.
\end{lem}
\begin{proof}
The half-plane $\Hplow$ is invariant by translation along $(\sin\frac{\alpha}{2},\cos\frac{\alpha}{2})$. Using this translation, we get that all the operators $\left(P_{\Add,\, \Hplow}+ V_{\B, \, \tau}\right)_{\tau\in\R}$ are unitary equivalent and their spectrum does not depend on $\tau$. Using a Fourier integral decomposition we have 
 $$P_{{\bf A}, \, \Hslow}=\int_{\tau\in \R}^{\bigoplus} P_{\Add,\, \Hplow}+ V_{\B, \, \tau} \d \tau \ , $$
 where $\Hslow$ is the half-space $\Hplow\times \R$. The normal of the boundary of $\Hslow$ is $(-\sin\frac{\alpha}{2},-\cos\frac{\alpha}{2},0)$. Therefore we have 
 $$\inf\spec\left(P_{\Add,\, \Hplow}+ V_{\B, \, \tau}\right) = \inf \spec \left(P_{{\bf A}, \, \Hslow}\right) \ .$$
  By an elementary computation we check that the magnetic field ${\bf B}$ makes the angle $\beta:=\arcsin(\sin\gamma\sin\alpha)$ with the boundary of $\Hslow$. Using the definition \eqref{D:sigma}, we get that the bottom of the spectrum of $P_{{\bf A}, \, \Hslow}$ is $\sigma(\beta)$.
\end{proof}

\section{Essential spectrum of the operators on the sector}
\label{S:specess}
Let $$\Upsilon:=V_{\B, \, \tau}^{-1}(\{0\}) \ \ \mbox{be the line where the electric potential vanish.}$$ Let us notice that $V_{\B, \, \tau}(x)$ is the square of the distance between $x$ and $\Upsilon$, moreover when ${\bf B}$ is tangent to a face of the wedge, the line $\Upsilon$ is parallel to one of the boundary of the sector $\Secteur_{\alpha}$. Since the domain is unbounded and the electric potential does not blow up in all directions, one should expect that the essential spectrum is not empty (see \cite[proposition 3.7]{HeMo02} for a similar situation). We denote by $\specess(P_{\Add,\, \Secteur_{\alpha}}+ V_{\B, \, \tau})$ the essential spectrum of $P_{\Add,\, \Secteur_{\alpha}}+ V_{\B, \, \tau}$ and we are looking for:
$$\ssess({\bf B};\alpha,\tau)=\inf \specess(P_{\Add,\, \Secteur_{\alpha}}+ V_{\B, \, \tau}) \ . $$
When the magnetic field is tangent to the edge, we use the results recalled in Subsection \ref{SS:opmodelewedgetangent} and we get $\ssess({\bf B},\alpha,\tau)=\Theta_{0}+\tau^2$. We will now assume that the magnetic field is not tangent to the edge, that is $\gamma\neq0$ where $\gamma$ is the first spherical coordinate of ${\bf B}$ (see \eqref{D:Bcoordpolar}). We recall a useful criterion for the characterization of the essential spectrum (see \cite{Pers60}):
\begin{lem}
We have 
$$  \ssess({\bf B};\alpha,\tau)=\lim_{R\to +\infty} \Sigma \left(P_{\Add,\, \Secteur_{\alpha}}+ V_{\B, \, \tau},R\right)$$
with
$$\Sigma \left(P_{\Add,\, \Secteur_{\alpha}}+ V_{\B, \, \tau},R\right):=\inf_{u\in \mathcal{C}_{0}^{\infty}(\overline{\Secteur_{\alpha}}\cap \complement\mathcal{B}_{R})}\frac{\FQ(u)}{\| u\|^2_{L^2(\Secteur_{\alpha})}} $$
where $\mathcal{B}_{R}$ is the ball of radius $R$ centered at the origin and $\complement\mathcal{B}_{R}$ its complementary in $\R^2$.
\end{lem}
%
\begin{prop}
\label{P:specess}
Assume that the magnetic field writes ${\bf B}=(\sin\gamma \cos\tfrac{\alpha}{2},\sin\gamma \sin\tfrac{\alpha}{2},\cos\gamma)$. We have: 
\begin{equation}
\label{F:infsurxi2}
\ssess({\bf B};\alpha,\tau)=\inf_{\xi_{2}\in\R}\left(\mudg_{1}(\xi_{2}\cos\gamma+\tau\sin\gamma)+(\xi_{2}\sin\gamma-\tau\cos\gamma)^2\right) \ . 
\end{equation}
\end{prop}
\begin{proof}
We show that $\ssess({\bf B};\alpha,\tau)=\inf \spec \left( P_{\Add,\, \Hpup}+ V_{\B, \, \tau}\right)$:
\\
{\sc Upper bound.}
Let $\epsilon>0$. Using the min-max principle we find a normalized function $u_{\epsilon}\in \mathcal{C}_{0}^{\infty}(\overline{\Hpup})$ such that $$\langle \left(P_{\Add,\, \Hpup}+ V_{\B, \, \tau}\right)u_{\epsilon},u_{\epsilon} \rangle_{L^2(\Hpup)}<\inf \spec\left(P_{\Add,\, \Hpup}+ V_{\B, \, \tau}\right)+\epsilon \ . $$ Let ${\bf t}_{\alpha}=(\cos\frac{\alpha}{2},\sin\frac{\alpha}{2})$ be the direction vector of the line $\Upsilon$ and for $r>0$ let $u_{\epsilon,r}(x):=e^{\frac{i}{2}rb_{3}{\bf t}_{\alpha}\wedge x}u_{\epsilon}(x-r{\bf t}_{\alpha})$. Let $R>0$, we have $\supp(u_{\epsilon,r})=\supp(u_{\epsilon})+r{\bf t}_{\alpha}$ and therefore it exists $r_{0}>0$ such that 
$\forall r>r_{0}, \quad \supp(u_{\epsilon,r})\subset \Secteur_{\alpha}\cap \complement \mathcal{B}_{R} \  $ and $u_{\epsilon,r}\in \dom(P_{\Add,\, \Secteur_{\alpha}}+ V_{\B, \, \tau})$.  We have $V_{\B, \, \tau}(x-r{\bf t}_{\alpha})=V_{\B, \, \tau}(x)$ hence from the translation principle we have 
 $$\FQ(u)=\langle \left(P_{\Add,\, \Secteur_{\alpha}}+ V_{\B, \, \tau}\right)u_{\epsilon,r},u_{\epsilon,r} \rangle_{L^2(\Secteur_{\alpha})}=\langle \left(P_{\Add,\, \Hpup}+ V_{\B, \, \tau}\right)u_{\epsilon},u_{\epsilon} \rangle_{L^2(\Hpup)} \ .$$
We deduce from the Persson's Lemma that $\ssess({\bf B};\alpha,\tau)\leq\inf \spec \left( P_{\Add,\, \Hpup}+ V_{\B, \, \tau}\right)$.
\\
{\sc Lower bound.} We denote by $(\rho,\phi)$ the polar coordinates of $\R^2$. Let $\chi_{1}^{\rm pol}$ and $\chi_{2}^{\rm pol}$ in $\mathcal{C}^{\infty}(\Secteur_{\alpha})$ that satisfy $0\leq\chi_{j}^{\rm pol} \leq 1$ and $\chi_{j}^{\rm pol}(r,\phi)=\chi_{j}^{\rm pol}(1,\phi)$. We assume that $\chi_{1}^{\rm pol}$ satisfies $\chi_{1}^{\rm pol}(r,\phi)=1$ when $\phi\in(\tfrac{\alpha}{4},\tfrac{\alpha}{2})$ and $\chi_{1}^{\rm pol}(r,\phi)=0$ when $\phi\in(-\tfrac{\alpha}{2},-\tfrac{\alpha}{4})$. We assume that $\chi_{2}^{\rm pol}$ satisfies $(\chi_{1}^{\rm pol})^2+(\chi_{2}^{\rm pol})^2=1$ and we denote by $\chi_{1}$ and $\chi_{2}$ the associated functions in cartesian coordinates. By construction for all $x\in \Secteur_{\alpha}$ we have $\chi_{j}(x)=\chi_{j}\big(\frac{x}{\| x\|}\big)$. We deduce:
$$\forall j\in\{1,2\}, \exists C_{0}, \forall R>0,  \forall x\in \Secteur_{\alpha}\cap\complement\mathcal{B}_{R}, \quad |\nabla\chi_{j}(x)|^2 \leq \frac{C_{0}}{R^2} \ . $$
Let $u\in \mathcal{C}_{0}^{\infty}(\overline{\Secteur_{\alpha}})$, the IMS formula (see \cite{CyFrKiSi87}) provides 
$$\FQ(u)=\sum_{j}\FQ(\chi_{j}u)-\sum_{j}\| \nabla\chi_{j}u\|^2  \ .$$
Since $\supp(\chi_{1}u)\subset \Hpup$ we have $\FQ(\chi_{1}u)\geq\inf\spec\left(P_{\Add,\, \Hpup}+ V_{\B, \, \tau}\right) \|\chi_{1}u \|^2_{L^2(\Secteur_{\alpha})}$. On the other part, elementary computations give $R_{0}>0$ such that for $R>R_{0}$ we have $\dist(\supp(\chi_{2}u),\Upsilon)=|R\sin\frac{\alpha}{4}\sin\gamma+\tau|$, therefore:
$$\forall R>R_{0},\forall x\in \supp(\chi_{2}u), \quad V_{{\B}, \, \tau}(x)\geq |R\sin\tfrac{\alpha}{4}\sin\gamma+\tau|^2$$
and for $R>R_{0}$ we get $\FQ(\chi_{2}u)\geq |R\sin\frac{\alpha}{4}+\tau|^2\| \chi_{2}u\|^2$. We deduce that for $R>R_{0}$:
$$\Sigma(P_{\Add,\, \Secteur_{\alpha}}+ V_{\B, \, \tau},R) \geq \inf\spec\left(P_{\Add,\, \Hpup}+ V_{\B, \, \tau}\right)-\frac{C_{0}}{R^2} $$
and we deduce $\ssess({\bf B};\alpha,\tau)\geq\inf \spec \left( P_{\Add,\, \Hpup}+ V_{\B, \, \tau}\right)$ from Persson's Lemma. We conclude using Lemma \eqref{L:Opmodeledemiplanup}.
\end{proof}

 We have an Agmon estimate for any eigenfunction associated to an eigenvalue below the essential spectrum. 
\begin{cor}
\label{C:expdecay}
Let ${\bf B}$ be a magnetic field tangent to a face of the wedge and $(\lambda,u_{\lambda})$ an eigenpair of $P_{\Add, \, \Secteur_{\alpha}}+V_{\B, \, \tau}$ such that $\lambda<\ssess({\bf B};\alpha,\tau)$. We have 
$$\forall \eta \in (0,\sqrt{\ssess({\bf B};\alpha,\tau)-\lambda}), \exists C>0, \quad \FQ(e^{\eta \Phi}u_{\lambda})<C\|u_{\lambda} \|_{L^2(\Secteur_{\alpha})}  $$
with $\Phi(x_{1},x_{2})=\sqrt{x_{1}^2+x_{2}^2}$.
\end{cor}
\begin{proof}
We refer to the standard proof of \cite{Bon06} and \cite{BoDauPopRay12} for this Agmon estimate.
\end{proof}
\begin{prop}
\label{P:limspecess}
We have 
$$\lim_{\tau\to+\infty}\ssess({\bf B};\alpha,\tau)=1 \ . $$
\end{prop}
\begin{proof}
For $\tau\geq0$, we take $\xi_{2}=\tau\cot \gamma$ in \eqref{F:infsurxi2} and we get 
$$\ssess({\bf B};\alpha,\tau) \leq \mudg_{1}(\tfrac{\tau}{\sin\gamma})<1 \ . $$
For the lower bound, we use \eqref{F:infsurxi2} and we make the distinction between two zones for $\xi_{2}$: 
\begin{itemize}
\item If $\xi_{2}\notin[\tau\cot\gamma-\tfrac{1}{\sin\gamma},\tau\cot\gamma+\tfrac{1}{\sin\gamma}]$, we get $(\xi_{2}\sin\gamma-\tau\cos\gamma)^2 \geq 1$.
\item If $\xi_{2}\in[\tau\cot\gamma-\tfrac{1}{\sin\gamma},\tau\cot\gamma+\tfrac{1}{\sin\gamma}]$, we have $\xi_{2}\cos\gamma+\tau\sin\gamma\in I_{\tau}$ with $I_{\tau}=[\frac{\tau}{\sin\gamma}-\cot\gamma,\frac{\tau}{\sin\gamma}+\cot\gamma]$. For $\tau$ large enough we have $I_{\tau}\subset (\xi_{0},+\infty)$. Since $\mudg_{1}$ is increasing on $(\xi_{0},+\infty)$, we get $\tau_{0}>0$ such that for all $\tau>\tau_{0}$: 
$$\forall \xi_{2}\in [\tau\cot\gamma-\tfrac{1}{\sin\gamma},\tau\cot\gamma+\tfrac{1}{\sin\gamma}], \quad \mudg_{1}(\xi_{2}\cos\gamma+\tau\sin\gamma) \geq \mudg_{1}(\tfrac{\tau}{\sin\gamma}-\cot\gamma) \ . $$
We conclude by using \eqref{F:infsurxi2} and the fact that $\mudg_{1}(\tau)$ tends to 1 as $\tau$ goes to $+\infty$.
\end{itemize}
\end{proof}
\begin{theo}
\label{T:majorationpartheta0}
Let ${\bf B}$ a magnetic field tangent to a face of the wedge $\Wedge_{\alpha}$. We have 
$$\sd({\bf B};\Wedge_{\alpha}) \leq \Theta_{0}\ .$$
\end{theo}
\begin{proof}
We choose $\tau=\xi_{0}\sin\gamma$ where $\xi_{0}$ is the unique point where $\mudg_{1}$ reaches its infimum (see subsection \ref{SS:demiplan}). Thanks to the proposition \ref{P:specess} we get $\ssess({\bf B};\alpha,\xi_{0}\sin\gamma)=\mudg_{1}(\xi_{0})=\Theta_{0}$ and we conclude using \eqref{E:relationfond}.
\end{proof}

\section{Limit when the Fourier parameter gets large}
\label{S:Limit}
In this section we investigate the limits of $\sse({\bf B};\alpha,\tau)$ when the Fourier parameter $\tau$ goes to $-\infty$ and $+\infty$. In the special case $\alpha=\frac{\pi}{2}$, Pan has identified these limits as eigenvalues of a model problem on a half-space and has given upper and lower bounds (see \cite{Pan02}). We provide an expression of these limits in the general case using the function $\sigma$ defined in \eqref{D:sigma}.
 Let ${\bf B}$ be a magnetic field of the form \eqref{D:Bcoordpolar}. Since 
$$\lim_{\tau\to-\infty}\left(\min_{(x_{1},x_{2})\in \Secteur_{\alpha}} V_{\B, \, \tau} (x_{1},x_{2}) \right)= +\infty \ , $$
we have from the min-max principle: 
$$ \lim_{\tau\to -\infty} \sse({\bf B};\alpha,\tau)=+\infty \ .$$
When $\tau$ goes to $+\infty$ the situation is much more different: $\Upsilon\cap \Secteur_{\alpha}$ is a half line which makes an angle $\alpha\in (0,\pi)$ with the boundary $\{ x_{2}=-\tan\frac{\alpha}{2}\}$ of $\Secteur_{\alpha}$. Moreover one should expect that any eigenfunction with energy below the essential spectrum is localized near the line $\Upsilon$.  In this situation we expect that $\sse({\bf B};a,\tau)$ tends to a quantity coming from a problem on regular domain when $\tau$ tends to $+\infty$. 
\begin{prop}
\label{P:limittaularge}
Assume that the magnetic field writes ${\bf B}=(\sin\gamma \cos\frac{\alpha}{2},\sin\gamma \sin\frac{\alpha}{2},\cos\gamma)$. Then we have 
$$\lim_{\tau\to+\infty} \sse({\bf B};\alpha,\tau)=\sigma(\beta)$$
with $\beta=\arcsin(\sin\alpha\sin\gamma)$.
\end{prop}
\begin{proof}
Thanks to Lemma \ref{L:Opmodeledemiplanlow}, for $\epsilon>0$ it exists $u_{\epsilon}\in \mathcal{C}_{0}^{\infty}(\overline{\Hplow})\cap \dom(P_{\Add,\, \Hplow}+ V_{\B, \, \tau})$ such that $\langle (P_{\Add,\, \Hplow}+ V_{\B, \, \tau})u_{\epsilon},u_{\epsilon} \rangle_{L^2(\Hplow)} < \sigma(\beta)+\epsilon$. We construct the test function
$$v_{\epsilon, \, \tau}(x):=e^{i\frac{\tau}{2} x \wedge {\bf t}_{\alpha}^{-}}u_{\epsilon}(x-\tau{\bf t}_{\alpha}^{-}) \ ,
$$
where ${\bf t}_{\alpha}^{-}=(\cos\frac{\alpha}{2},-\sin\frac{\alpha}{2})$ is the direction of the lower boundary of $\Secteur_{\alpha}$. For $\tau$ large enough, we have $\supp(v_{\epsilon, \, \tau}) \subset \Secteur_{\alpha}$ and thus $v_{\epsilon, \,  \tau}\in \dom(P_{\Add,\, \Secteur_{\alpha}}+ V_{\B, \, \tau})$. From the translation principle we get 
$$ \langle (P_{\Add,\, \Secteur_{\alpha}}+ V_{\B, \, \tau})v_{\epsilon, \, \tau},v_{\epsilon, \, \tau} \rangle_{L^2(\Secteur_{\alpha})}=\langle (P_{\Add,\, \Hplow}+ V_{\B, \, \tau})u_{\epsilon},u_{\epsilon} \rangle_{L^2(\Hplow)} < \sigma(\beta)+\epsilon  $$
and we deduce from the min-max principle that 
\begin{equation}
\label{E:majorationsenpinfini}
\limsup_{\tau\to+\infty}\sse({\bf B};\alpha,\tau) \leq \sigma(\beta) \ .
\end{equation} 
When $\alpha=\frac{\pi}{2}$ the proposition has already been proved in \cite{Pan02}. We now suppose that $\alpha\neq\frac{\pi}{2}$ and thus $\beta\neq\frac{\pi}{2}$. Using Proposition \ref{P:limspecess} and the fact that $\forall  \beta\in(0,\frac{\pi}{2}), \sigma(\beta) < 1$, we get that for $\tau$ large enough, $\sse({\bf B};\alpha,\tau)$ is an eigenvalue of $P_{\Add,\, \Secteur_{\alpha}}+ V_{\B, \, \tau}$ with finite multiplicity. We denote by $u_{\tau}$ an associated eigenfunction. To establish a lower bound for $\sse({\bf B};\alpha,\tau)$, we use the concentration of the  eigenfunctions near the line $\Upsilon$ and an IMS formula.  Let $(\chi_{j})_{j\in\{1,2,3\}}\in C^{\infty}(\R)$ such that $0 \leq \chi_{j} \leq 1$ and
\begin{equation*}
\left \{
\begin{aligned}
&\chi_{1}=1 \mbox{ on } (-\infty,-\tfrac{1}{2}] \ \mbox{ and } \  \chi_{1}=0 \mbox{ on } [-\tfrac{1}{4},+\infty) \ ,
\\
&\chi_{2}=1 \mbox{ on } (-\tfrac{1}{4},\tfrac{1}{4}] \ \mbox{ and } \  \chi_{2}=0 \mbox{ on } (-\infty,-\tfrac{1}{2}]\cup [\tfrac{1}{2},+\infty) \ ,
\\
&\chi_{3}=0 \mbox{ on } (-\infty,\tfrac{1}{4}] \ \mbox{ and } \  \chi_{3}=1 \mbox{ on } [\tfrac{1}{2},+\infty) \ ,
\\
&\sum_{j=1}^3\chi_{j}^2=1 \ .
\end{aligned}
\right.
\end{equation*}
We define for $j\in\{1,2,3\}$:
$$\chi_{j,\tau}(x_{1},x_{2}):=\chi_{j}\left(\tau^{-1}(x_{1}b_{2}-x_{2}b_{1}-\tau)\right) \ . $$
 Since the magnetic field is non tangent to the edge, $b_{1}$ or $b_{2}$ is non-zero and it exists $C>0$ and $\tau_{1}>0$ such that 
\begin{equation}
\forall \tau \geq \tau_{1}, \, \forall j\in\{1,2,3\},\,  \forall (x_{1},x_{2})\in \Secteur_{\alpha}, \quad |\nabla \chi_{j,\, \tau}(x_{1},x_{2})|^2 \leq \frac{C}{\tau^2} \ . 
\end{equation}
Using the IMS formula we get: 
$$\FQ(u_{\tau})=\sum_{j}\FQ(\chi_{j,\, \tau}u_{\tau})-\sum_{j}\| \nabla\chi_{j,\, \tau}u_{\tau}\|^2  \ .$$
Let $\epsilon>0$. It exists $\tau_{1}$ such that we have 
\begin{equation}
\label{m:sanslescutoff}
\forall \tau \geq \tau_{1}, \quad \FQ(u_{\tau}) \geq \FQ(\chi_{2,\, \tau}u_{\tau})-3\epsilon \ . 
\end{equation}
Since $\supp(\chi_{2,\, \tau})\cap \partial\Secteur_{\alpha}\subset \{x_{2}=-x_{1}\tan\frac{\alpha}{2}\}$, we extend $\chi_{2, \, \tau}u_{\tau}$ to a function of $\dom\left(P_{\Add,\, \Hplow}+ V_{\B, \, \tau}\right)$ which satisfy the Neumann boundary condition by taking the value 0 outside $\supp(\chi_{2,\, \tau}u_{\tau})$. Therefore using Lemma \ref{L:Opmodeledemiplanlow} we get 
\begin{equation}
\label{m:energieprincipale}
\FQ(\chi_{2,\, \tau}u_{\tau})=\langle (P_{\Add,\, \Hplow}+ V_{\B, \, \tau})\chi_{2, \, \tau}u_{\tau},\chi_{2, \, \tau}u_{\tau} \rangle_{L^2(\Hplow)} \geq \sigma(\beta) \| \chi_{2,\, \tau}u_{\tau}\|^2_{L^2(\Hplow)} \ . 
\end{equation}
 For $\tau$ large enough we have from \eqref{E:majorationsenpinfini}:
\begin{equation}
\label{M:energiecinetique}
\int_{\supp(\chi_{j,\, \tau})} V_{\B, \, \tau}|u_{\tau}|^2\d x_{1} \d x_{2} \leq \FQ(u_{\tau}) \leq \sigma(\beta)+\epsilon  \ .
\end{equation}
When $j\in\{1,3\}$, $\supp(\chi_{j})\cap \Secteur_{\alpha}\subset \{ x\in \Secteur_{\alpha}, \dist(x,\Upsilon) \geq \frac{\tau}{2}\}$. 
We deduce
$$\forall j\in \{1,3\}, \quad \lim_{\tau\to+\infty} \left(\inf_{\supp(\chi_{j,\,\tau})} V_{\B, \, \tau}\right)=+\infty $$
and due to \eqref{M:energiecinetique}:
$$\forall j\in \{1,3\}, \quad  \lim_{\tau\to+\infty} \int_{\supp(\chi_{j, \, \tau})}|u_{\tau}|^2\d x_{1} \d x_{2}=0 \ . $$
We deduce that it exists $\tau_{2}>\tau_{1}$ such that  
$$\forall \tau \geq \tau_{2}, \, \forall j\in \{1,3 \},\quad \| \chi_{j, \, \tau}u_{\tau}\|^2_{L^2(\Secteur_{\alpha})} \leq \epsilon \ $$
and using that $\|u_{\tau}\|_{L^2(\Secteur_{\alpha})}=1$: 
$$\forall \tau \geq \tau_{2}, \quad \| \chi_{2, \, \tau}u_{\tau} \|^2_{L^2(\Secteur_{\alpha})} \geq 1-2\epsilon \ . $$ 
Using \eqref{m:sanslescutoff} and \eqref{m:energieprincipale} we get for $\tau \geq \tau_{2}$:
$$\FQ(u_{\tau}) \geq \sigma(\beta)(1-2\epsilon)-3\epsilon \ .$$
Since $u_{\tau}$ is normalized, we get $$\liminf_{\tau\to+\infty}\sse({\bf B};\alpha,\tau) \geq \sigma(\beta)  $$
and the proposition is proved.
\end{proof}
Using Theorem \ref{T:majorationpartheta0}, we deduce the following: 
\begin{cor}
\label{C:minatteint}
Assume that the magnetic field writes ${\bf B}=(\sin\gamma \cos\frac{\alpha}{2},\sin\gamma \sin\frac{\alpha}{2},\cos\gamma)$. Then the function $\tau\mapsto\sse({\bf B};\alpha,\tau)$ reaches its infimum.
\end{cor}

\section{Rough upper bounds}
\label{S:ub}
In this Section we provide an upper bound for $\spectre({\bf B};\Wedge_{\alpha})$ using quasi-modes from \cite{Bon06}.
\begin{prop}
Assume that the magnetic field writes ${\bf B}=(\sin\gamma\cos\frac{\alpha}{2},\sin\gamma\sin\frac{\alpha}{2},\cos\gamma)$. Then
\begin{equation}
\label{E:lowerbound}
\forall (\alpha,\gamma)\in (0,\pi)\times [0,\tfrac{\pi}{2}], \quad \spectre({\bf B};\Wedge_{\alpha})\leq \alpha \left(\frac{1}{\sqrt{3}}+\frac{\sqrt{3}}{2}\sin^2\gamma \right) \ .
 \end{equation}
\end{prop}
\begin{proof} 
We set $\tau=0$ and we make several standard transformations in the quadratic form $\FQ$ in order to study a quadratic form on a domain independent from $\alpha$  (see \cite[Section 3]{Bon06} or \cite[Section 5.1]{Tbon} for the details). We start with a change of variables associated with the polar coordinates $(\rho,\phi)\in \Omega^{\alpha}$ with $\Omega^{\alpha}=\R_{+}\times (-\frac{\alpha}{2},\frac{\alpha}{2})$ and we are led to the quadratic form
$$v\mapsto \int_{\Omega^{\alpha}} \left(|\partial_{\rho}v|^2+\frac{1}{\rho^2}|(\partial_{\phi}+i\frac{\rho^2}{2})v|^2+\Vpol_{\B} |v|^2\right)\rho \d \rho \d \phi $$
 with
  \begin{equation}
\label{}
\Vpol_{{\underline{\bf B}}}(\rho,\eta):=\big(\rho\cos(\eta\alpha)b_{2}-\rho\sin(\eta\alpha)b_{1}\big)^2 \ .
\end{equation}
the electric potential in polar coordinates
We make the change of gauge 
$$ u(\rho,\phi):=e^{i\frac{\rho^2}{2}\phi}v(\rho,\phi)$$
and we normalize the angle with the scaling $\eta=\frac{\phi}{\alpha}$. Using these transformations  we get that for $\tau=0$ the quadratic form $\FQ$ is unitary equivalent to the quadratic form
\begin{equation}
\label{D:FQpol}
\FQpol(u):=\int_{\Omega_{0}}\left(|(\partial_{\rho}-i\alpha \rho \eta b_{3})u|^2+\frac{1}{\alpha^2 \rho^2}|\partial_{\eta}u|^2+\Vpol_{\underline{\bf B}}|u|^2\right)\rho \d \rho \d \eta 
\end{equation}
  with $\Omega_{0}=\R_{+}\times (-\frac{1}{2},\frac{1}{2})$.  
The form domain is 
\begin{equation*}
\label{Domaine_fqpol}
\dom(\FQpol)=
\left\{ u\in L^2_{\rho}(\Omega_{0}), \, (\partial_{\rho}-i\alpha \rho \eta b_{3})u\in L^2_{\rho}(\Omega_{0}),\,  \frac{1}{\rho}\partial_{\eta}u\in L^2_{r}(\Omega_{0}),\,  \sqrt{\Vpol_{\underline{\bf B}}}\, u\in L^2_{\rho}(\Omega_{0}) \right\}
\end{equation*}
where $L^2_{\rho}(\Omega_{0})$ stands for the set of the square-integrable functions for the weight $\rho\d \rho$. Let $B^1_{\rho}(\Rp) :=\{ u\in L^2_{\rho}(\R^+), u'\in L^2_{\rho}(\R^+), \rho u \in L^2_{\rho}(\R^+)\}$.
We have an injection from $B^1_{\rho}(\Rp)$ into $\dom(\FQpol)$, and for $u\in B^1_{\rho}(\Rp)$ an elementary computation (see \cite[Proposition 6.26]{Popoff}) yields: 
$$\FQpol(u)=\|u'\|_{L^2_{\rho}(\R_{+})}^2+\left(b_{2}^2+\frac{\alpha^2}{12}b_{3}^2+\frac{1}{2}(1-\sinc \alpha)(b_{1}^2-b_{2}^2) \right)\| \rho u \|_{L^2_{\rho}(\R_{+})}^2 \ ,$$
where $\sinc\alpha:=\frac{\sin\alpha}{\alpha}$. We take the quasimode $u_{\alpha}(\rho,\eta):=3^{-1/4}\exp\left(\frac{-\alpha \rho^2}{4\sqrt{3}}\right)$ coming from \cite{Bon06}. The function $u_{\alpha}$ is in $B_{r}^{1}(\R_{+})$. We get $$ \quad \|u_{\alpha}'\|_{L^2_{\rho}(\R_{+})}^2=\frac{1}{2\sqrt{3}} \quad \mbox{and} \quad \|\rho u_{\alpha}\|_{L^2_{\rho}(\R_{+})}^2=\frac{2\sqrt{3}}{\alpha^2} \ . $$
Using $(b_{1},b_{2},b_{3})=(\sin\gamma\cos\frac{\alpha}{2},\sin\gamma\sin\frac{\alpha}{2},\cos\gamma)$, we get 
$$\FQpol(u_{\alpha})=\frac{1}{2\sqrt{3}}+\frac{\sqrt{3}}{2}\left(\sinc\frac{\alpha}{2}\right)^2\sin^2\gamma+\frac{\cos^2\gamma}{2\sqrt{3}}+\sqrt{3}\frac{1-\sinc\alpha}{\alpha^2}\cos\alpha\sin^2\gamma \ . $$
Since $\sinc \frac{\alpha}{2} \leq 1$ and $0\leq\frac{1-\sinc \alpha}{\alpha^2}\leq \frac{1}{6}$, using $\|u_{\alpha}\|_{L^2_{\rho}(\R_{+})}^2=\frac{1}{\alpha}$ we get from the min-max principle: 
$$\forall \alpha\in (0,\pi], \quad \sse({\bf B};\alpha,0) \leq \alpha \left(\frac{1}{\sqrt{3}}+\frac{\sqrt{3}}{2}\sin^2\gamma \right) \ .$$
We conclude with the relation \eqref{E:relationfond}.
\end{proof}
When the magnetic field is tangent to a face of the wedge, we deduce:
\begin{equation}
\lim_{\alpha\to 0}\spectre({\bf B};\Wedge_{\alpha})=0.
\end{equation}
From Corollary \ref{C:minatteint} we know that the function $\tau\mapsto \sse({\bf B};\alpha,\tau)$ reaches its infimum when ${\bf B}$ is tangent to a face of $\Wedge_{\alpha}$. For $\alpha$ small enough, we are able to characterize the bottom of the spectrum of the operator on the sector, indeed using Proposition \ref{P:specess} and the lower bound \eqref{E:lowerbound} we get:
\begin{cor}
Assume that the magnetic field writes ${\bf B}=(\sin\gamma\cos\frac{\alpha}{2},\sin\gamma\sin\frac{\alpha}{2},\cos\gamma)$ and that $\alpha \left(\frac{1}{\sqrt{3}}+\frac{\sqrt{3}}{2}\sin^2\gamma \right)<\Theta_{0}$. Let  $\tau^{*}$ be a value of the parameter such that $\spectre({\bf B};\Wedge_{\alpha})=\sse({\bf B};\alpha,\tau^{*})$. Then $\sse({\bf B};\alpha,\tau^{*})$ is a discrete eigenvalue for $P_{\Add, \, \Secteur_{\alpha}}+V_{\B, \, \tau^{*}}$.
\end{cor}
\begin{rem}
The approximation \eqref{valeur-numerique-theta0} gives a precise set of values for $\alpha$ and $\gamma$ such that the condition in the previous corollary holds.
\end{rem}
\section{Particular case: a magnetic field normal to the edge}
\label{S:champnormal}
We assume here that the magnetic field ${\bf B}$ is tangent to a face and normal to the edge. Therefore its spherical coordinates are $(\gamma,\theta)=(\frac{\pi}{2},\frac{\pi-\alpha}{2})$ and its cartesian coordinates are $(\cos\frac{\alpha}{2},\sin\frac{\alpha}{2},0)$. In that case we have $\Add=0$ and the operator $P_{\Add,\, \Secteur_{\alpha}}+ V_{\B, \, \tau}$ writes 
$ -\Delta+V_{\B, \, \tau}$ with $V_{\B, \, \tau}=(x_{1}\sin\tfrac{\alpha}{2}-x_{2}\cos\tfrac{\alpha}{2}-\tau)^2$.
\begin{prop}
\label{P:monotonie}
Let ${\bf B}$ be a constant magnetic field of spherical coordinates $(\frac{\pi}{2},\frac{\pi-\alpha}{2})$. Then $\alpha\mapsto \sd({\bf B};\Wedge_{\alpha})$ is non-decreasing on $(0,\frac{\pi}{2}]$.
\end{prop}
\begin{proof}
Let $\alpha\in (0,\frac{\pi}{2}]$. The operator $P_{\Add,\, \Secteur_{\alpha}}+ V_{\B, \, \tau}$ writes 
$$-\Delta+(x_{1}\sin\tfrac{\alpha}{2}-x_{2}\cos\tfrac{\alpha}{2}-\tau)^2$$
in the sector $\Secteur_{\alpha}$. We denote by $R_{\omega}$ the rotation centered at the origin of angle $\omega$. We make the change of variables $(u_{1},u_{2}):=R_{-\frac{\alpha}{2}}(x_{1},x_{2})$. Since $\alpha\leq\frac{\pi}{2}$, we have $R_{-\frac{\alpha}{2}}(\Secteur_{\alpha})=\{(u_{1},u_{2})\in \R^2,u_{1}>0,  -u_{1}\tan\alpha \leq u_{2} \leq 0 \}$. In these variables the operator $P_{\Add,\, \Secteur_{\alpha}}+ V_{\B, \, \tau}$ becomes 
$$-\Delta+(u_{2}-\tau)^2 \ .$$
We make the dilatation $(v_{1},v_{2})=(-u_{1}\tan\alpha,u_{2})$ and the problem is unitary equivalent to the Neumann realization of
$$-(\tan\alpha)^2\partial_{v_{1}}^2-\partial_{v_{2}}^2+(v_{2}-\tau)^2 $$
in $\{v_{1}>0, v_{1}\leqÊv_{2} \leq 0 \}$. Using the min-max principle, we find that $\sse({\bf B};\alpha,\tau)$ is non-decreasing with $\alpha$ on $(0,\frac{\pi}{2})$ for all $\tau\in\R$. Using \eqref{E:relationfond} we get the proposition.
\end{proof}

The following result was already known by Pan for $\alpha=\frac{\pi}{2}$: 
\begin{theo}
\label{T:egalite}
Let ${\bf B}$ a constant magnetic field of spherical coordinates $(\frac{\pi}{2},\frac{\pi-\alpha}{2})$. Then 
$$\forall \alpha \in \left[\tfrac{\pi}{2},\pi\right], \quad \sd({\bf B};\Wedge_{\alpha})=\Theta_{0} \ . $$
Moreover $\sse({\bf B};\alpha,\tau)=\spectre({\bf B};\Wedge_{\alpha})$ if and only if $\tau=\xi_{0}$, and $\spectre({\bf B};\Wedge_{\alpha})=\ssess({\bf B};\alpha,\xi_{0})$.
\end{theo}
\begin{proof}
The upper bound comes from Theorem \ref{T:majorationpartheta0}. We will provide a lower bound (in the sense of the quadratic forms) for the operator 
$$-\Delta +(x_{1}\cos\tfrac{\alpha}{2}-x_{2}\sin\tfrac{\alpha}{2}-\tau)^2 \ .$$
 Let $\omega\in (0,2\pi)$. Making the change of variables $(u_{1},u_{2})=R_{\omega}(x_{1},x_{2})$, we get that the operator $-\Delta +(x_{1}\cos\frac{\alpha}{2}-x_{2}\sin\frac{\alpha}{2}-\tau)^2$ is unitary equivalent to the Neumann realization of 
$$-\partial_{u_{1}}^2-\partial_{u_{2}}^2+(u_{1}\cos(\tfrac{\alpha}{2}+\omega)-u_{2}\sin(\tfrac{\alpha}{2}+\omega)-\tau)^2, \quad (u_{1},u_{2})\in R_{\omega}(\Secteur_{\alpha}) \ . $$
We introduce two 1D-operators on the half-line with a Neumann boundary condition: 
$$L_{\rho,u_{2}}:=-\partial_{u_{1}}^2+\rho^2(u_{1}\cos(\tfrac{\alpha}{2}+\omega)-u_{2}\sin(\tfrac{\alpha}{2}+\omega)-\tau)^2, \quad u_{1}>0  \ ,$$
and
$$L_{\rho,u_{1}}:=-\partial_{u_{2}}^2+(1-\rho^2)(u_{1}\cos(\tfrac{\alpha}{2}+\omega)-u_{2}\sin(\tfrac{\alpha}{2}+\omega)-\tau)^2, \quad u_{2}>0 \ . $$
Since $\alpha\in[\frac{\pi}{2},\pi]$, we can choose $\omega\in(-\frac{\alpha}{2},-\frac{\pi}{2}+\frac{\alpha}{2})$ such that the two axes $\{u_{1}>0\}$ and $\{u_{2}>0\}$ belong to $\overline{R_{\omega}(\Secteur_{\alpha})}$. Therefore we have (in the sense of quadratic forms): 
$$-\partial_{u_{1}}^2-\partial_{u_{2}}^2+(u_{1}\cos(\tfrac{\alpha}{2}+\omega)-u_{2}\sin(\tfrac{\alpha}{2}+\omega)-\tau)^2 \geq L_{\rho,u_{2}}+L_{\rho,u_{2}} \ . $$
Due to an elementary scaling, we have: 
\begin{equation}
\label{Minop1d}
L_{\rho,u_{2}} \geq \rho\cos(\tfrac{\alpha}{2}+\omega) \Theta_{0} \quad  \mbox{ and } \quad L_{\rho,u_{1}}\geq \sqrt{1-\rho^2}\sin(\tfrac{\alpha}{2}+\omega)\Theta_{0} \ . 
\end{equation}
Therefore we have 
$$\forall \tau \in \R, \quad \sse({\bf B};\alpha,\tau) \geq \rho \cos(\tfrac{\alpha}{2}+\omega) \Theta_{0}+\sqrt{1-\rho^2}\sin(\tfrac{\alpha}{2}+\omega)\Theta_{0} \ . $$
We optimize the lower bound by taking $\rho=\cos(\frac{\alpha}{2}+\omega)$ and using \eqref{E:relationfond} we get $\sd({\bf B};\Wedge_{\alpha}) \geq \Theta_{0}$.
\end{proof}

\section{Numerical simulations}
\label{S:numericalsimu}
Numerically we compute the first eigenpair of the operator $P_{\Add, \, \Secteur_{\alpha}}+V_{\B, \, \tau}$ on the triangle $\mathcal{T}_{\alpha,L}:=\Secteur_{\alpha}\cap \{0 < x_{1}<L \}$ with a Dirichlet condition on the artificial boundary $\{ x_{1}=L\}$.
We use the finite element library M\'elina (\cite{Melina++}) and we refer to \cite[Section 4.4 and Annex C]{Popoff} for more details about the meshes and the degree of the elements we have used. We choose for the magnetic potential $\Add^{\rm R}(x_{1},x_{2}):=(-x_{1}b_{3},0)$.

On figures \ref{F1} and \ref{F2} we take $\alpha=\frac{\pi}{2}$ and a magnetic field of spherical coordinates $(\gamma,\theta)=(\frac{\pi}{10},\frac{\pi}{4})$ tangent to a face of the wedge. The computational domain is $\mathcal{T}_{\frac{\pi}{2},14}$. 

On figure \ref{F1} we show numerical approximation of the band function $\tau\mapsto \sse({\bf B};\alpha,\tau)$. We denote by $\breve{\sse}({\bf B};\alpha,\tau)$ these approximations. We have made the computations for $\tau=\frac{k}{10}$ with $-10\leq k\leq 18$. We have also plotted the bottom of the essential spectrum of the operator $P_{\Add, \, \Secteur_{\alpha}}+V_{\B, \, \tau}$ (according to the relation \eqref{F:infsurxi2}), the constant $\Theta_{0}$ and $\sigma(\beta)$ with $\beta=\arcsin(\sin\gamma\sin\frac{\alpha}{2})$. The numerical approximation of $\sigma(\beta)$ comes from \cite{BoDauPopRay12}.

We observe that $\breve{\sse}({\bf B};\alpha,\tau)<\ssess({\bf B};\alpha,\tau)$ and that  $\tau \mapsto \breve{\sse}({\bf B};\alpha,\tau)$ has a unique minimum. Moreover this minimum is smaller than $\Theta_{0}$. When $\tau$ goes to $+\infty$, $\breve{\sse}({\bf B};\alpha,\tau)$ tends to $\sigma(\beta)$ according to Proposition \ref{P:limittaularge}.

\begin{figure}[ht!]
\hspace{-1cm}
\includegraphics[keepaspectratio=true,width=16cm]{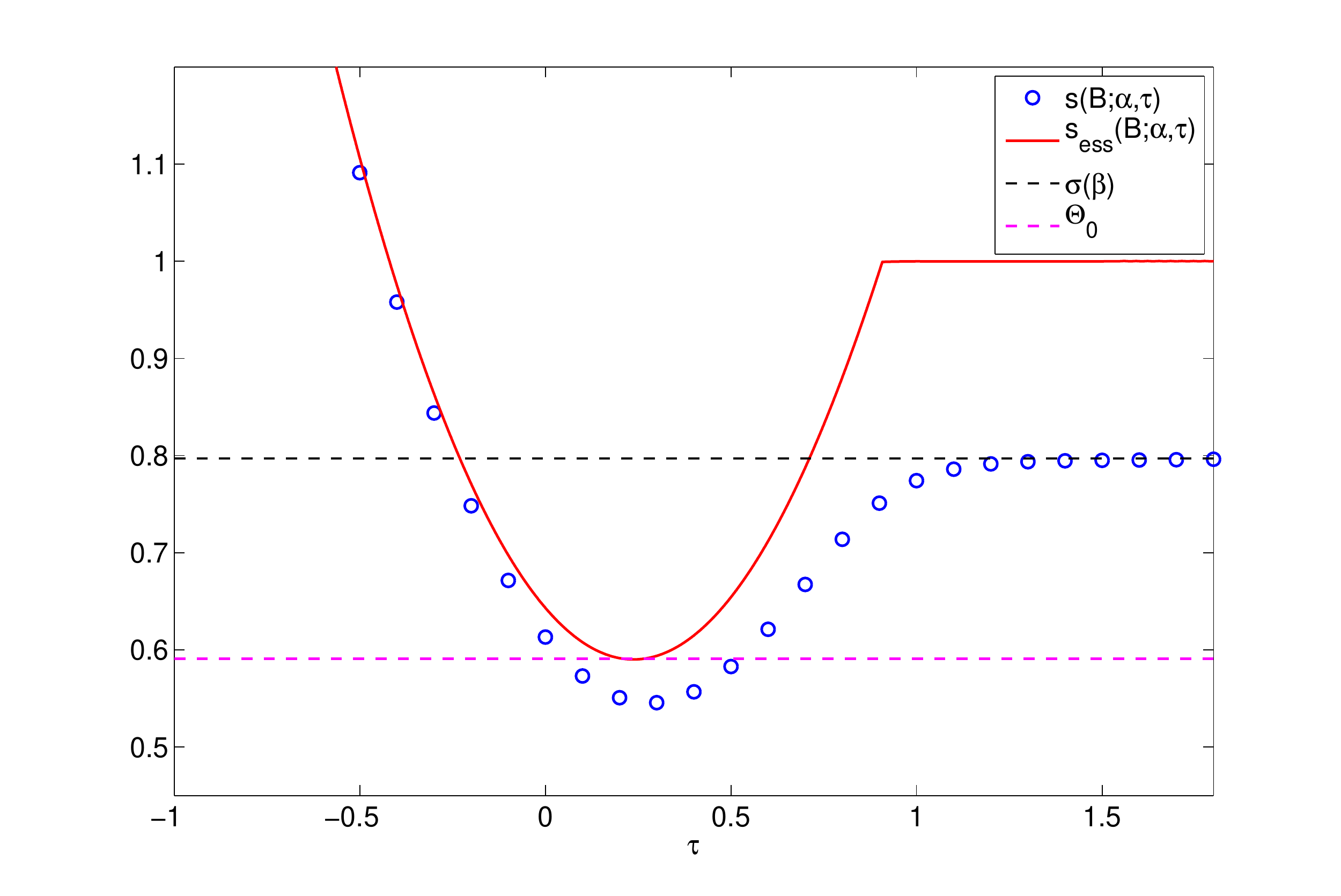}
\caption{Opening angle: $\alpha=\frac{\pi}{2}$. Spherical coordinates of ${\bf B}$: $(\gamma,\theta)=(\frac{\pi}{10},\frac{\pi}{4})$. The approximation $\breve{\sse}(\gamma,\theta;\alpha,\tau)$ with respect to $\tau$ for $\tau=\frac{k}{10}$, $-10\leq k\leq 18$ compared to $\ssess({\bf B};\alpha,\tau)$, $\Theta_{0}$ and $\sigma(\beta)$.  
}
\label{F1}
\end{figure}

On figure \ref{F2} we have plotted the eigenfunctions associated to the values of $\breve{\sse}({\bf B};\alpha,\tau)$ shown in figure \ref{F1} for $\tau=\frac{k}{2}$ with $0 \leq k \leq 4$. From top to bottom we show the modulus, the base-10 logarithm of the modulus and the phases modulo $\pi$ of the eigenfunctions. The logarithm is set to -13 when the value of the modulus is less than $10^{-13}$. The phases of an eigenfunction $u$ is computed according to the formula 
\begin{equation}
\label{R:C0phasemodulopi}
\phi(x_{1},x_{2}):=\arcsin \left(\frac{\Im\big(u(x_{1},x_{2})\big)}{|u(x_{1},x_{2})|}\right) \ .
\end{equation}
On the logarithm scale of the modulus we have shown in dash line the set $\Upsilon$ where the potential $V_{\B, \, \tau}$ vanishes. 

\newpage

\begin{figure}[h!]
\hspace{-0.2cm}
\begin{tabular}{cccccc}
\includegraphics[keepaspectratio=true,width=2.0cm]{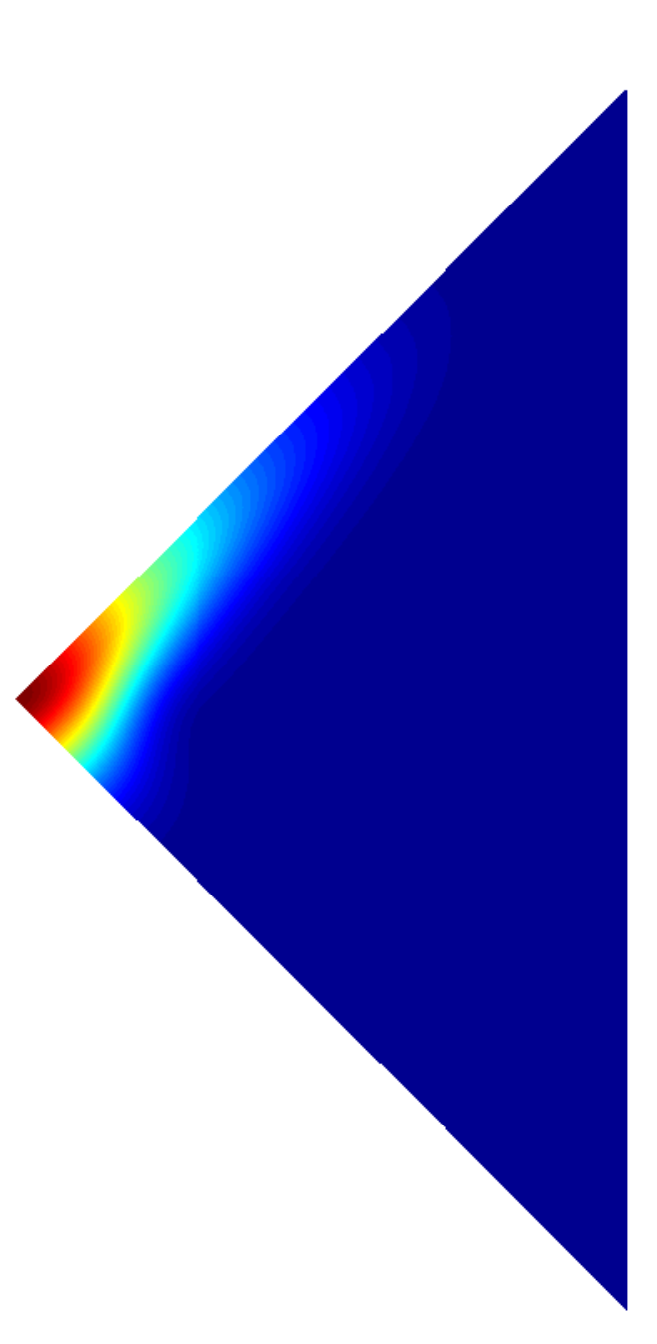}
&
\includegraphics[keepaspectratio=true,width=2.0cm]{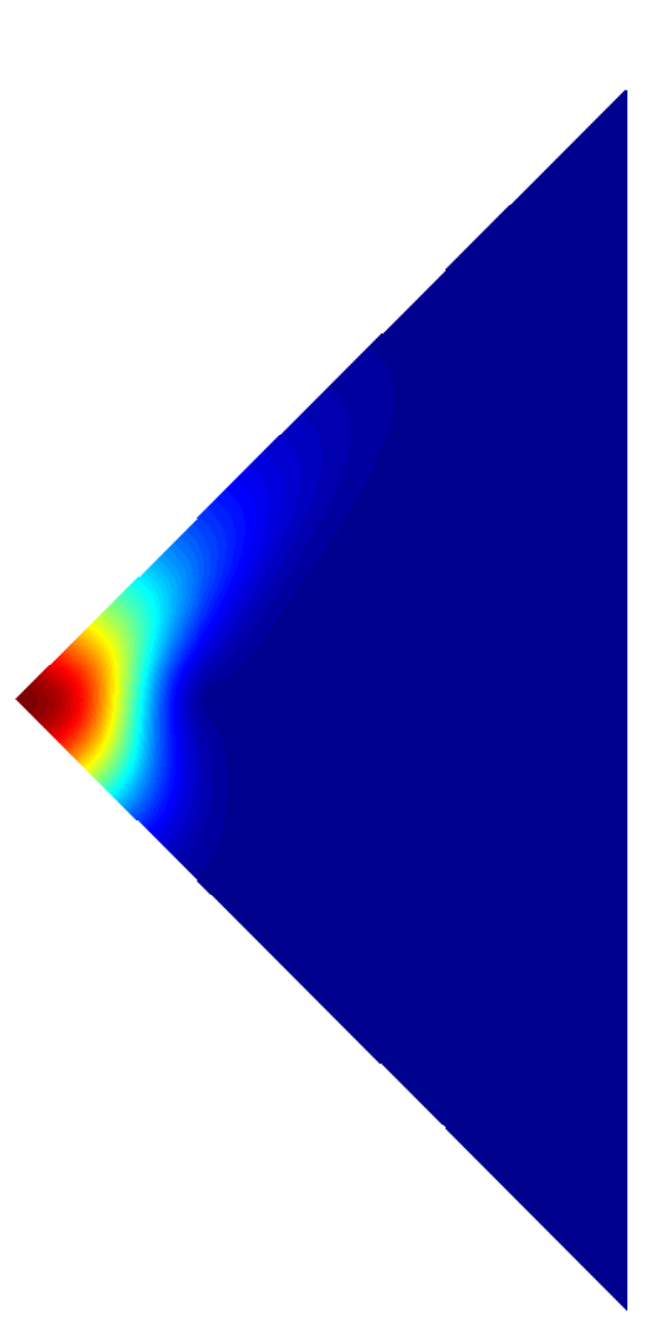}
  &
\includegraphics[keepaspectratio=true,width=2.0cm]{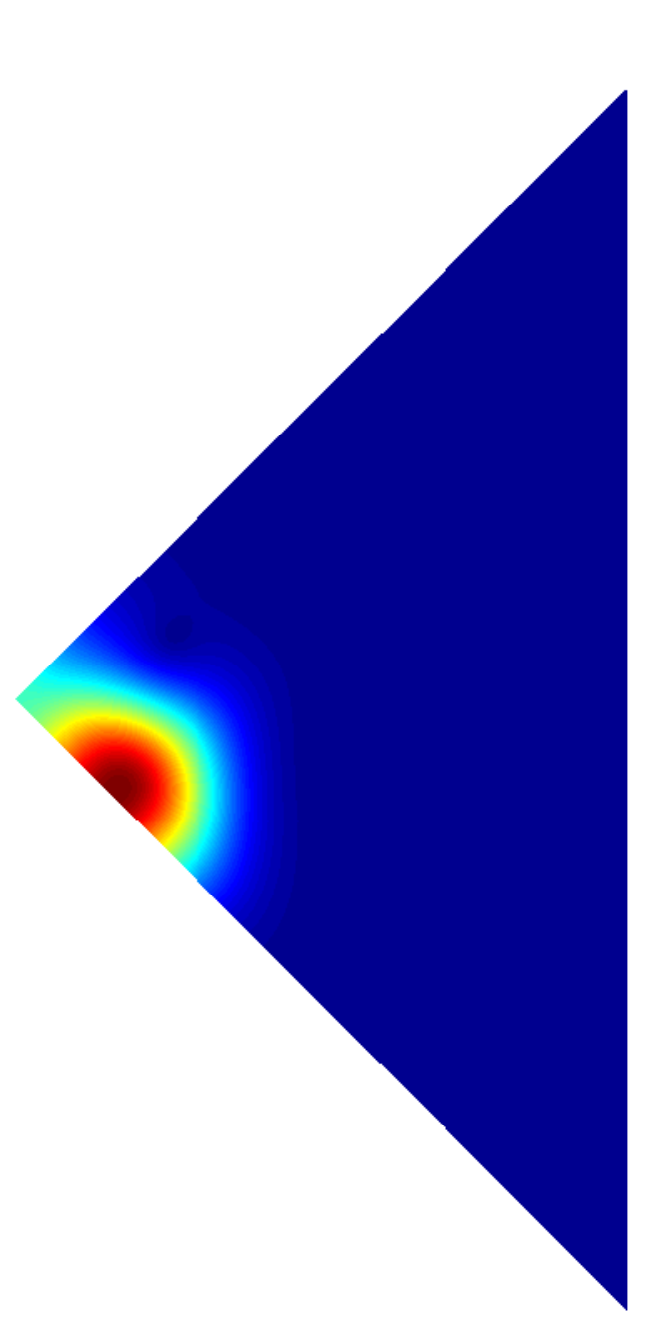}  
&
\includegraphics[keepaspectratio=true,width=2.0cm]{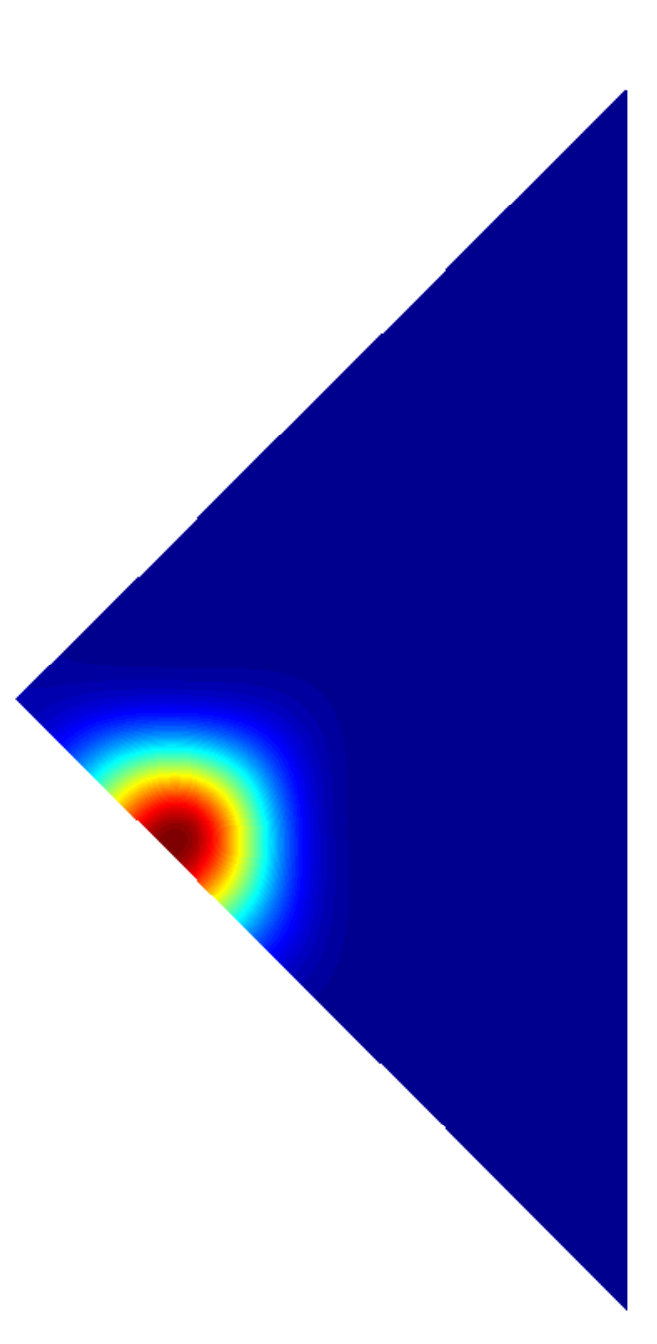}  
&
\includegraphics[keepaspectratio=true,width=2.0cm]{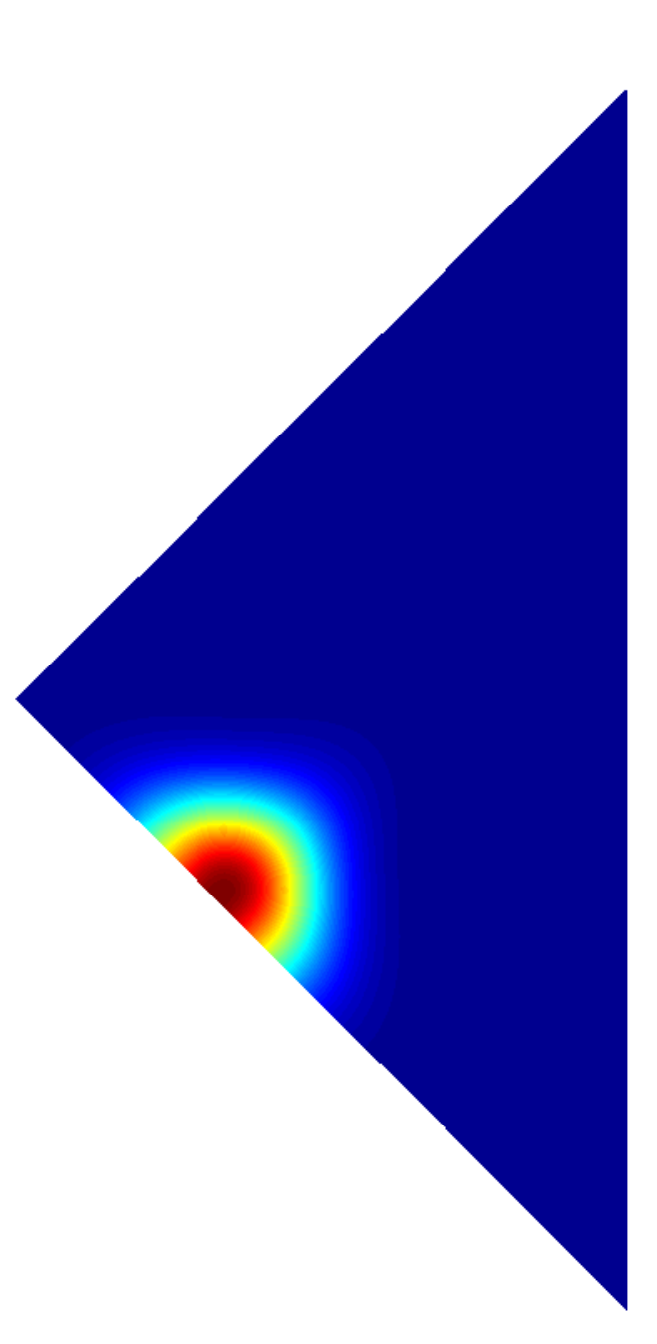}  
&
\includegraphics[keepaspectratio=true,width=0.61cm]{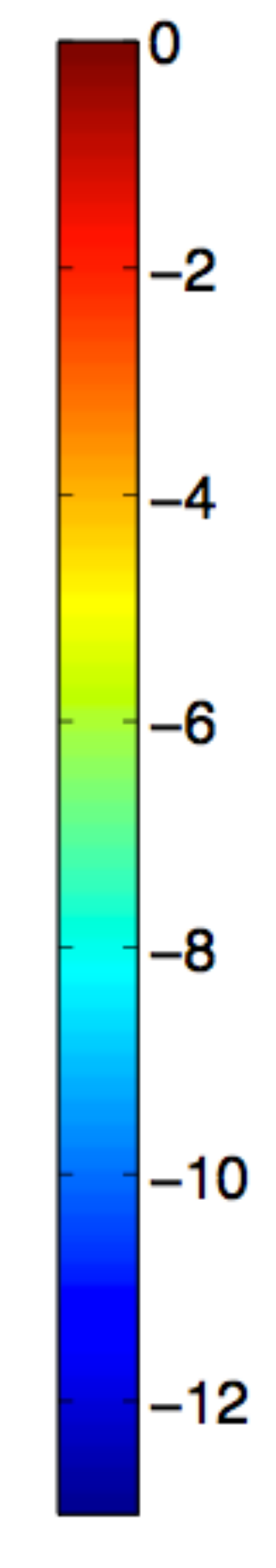} 
\\
\includegraphics[keepaspectratio=true,width=2.0cm]{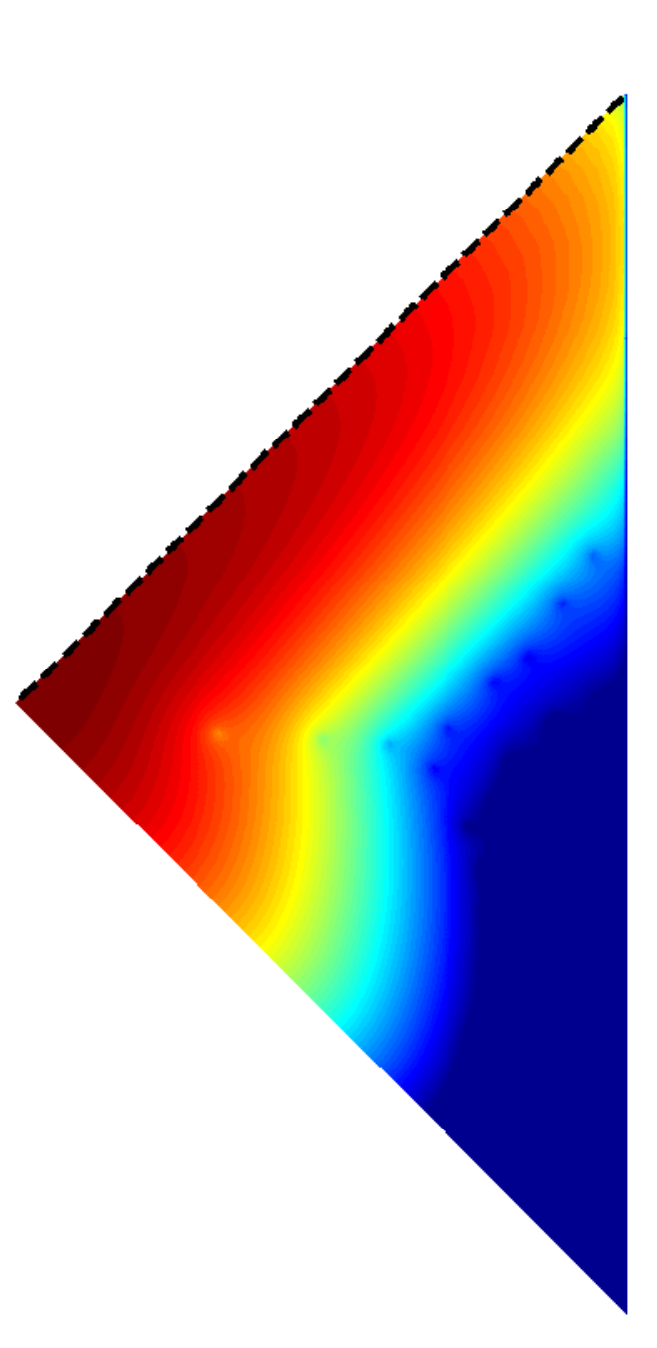}
&
\includegraphics[keepaspectratio=true,width=2.0cm]{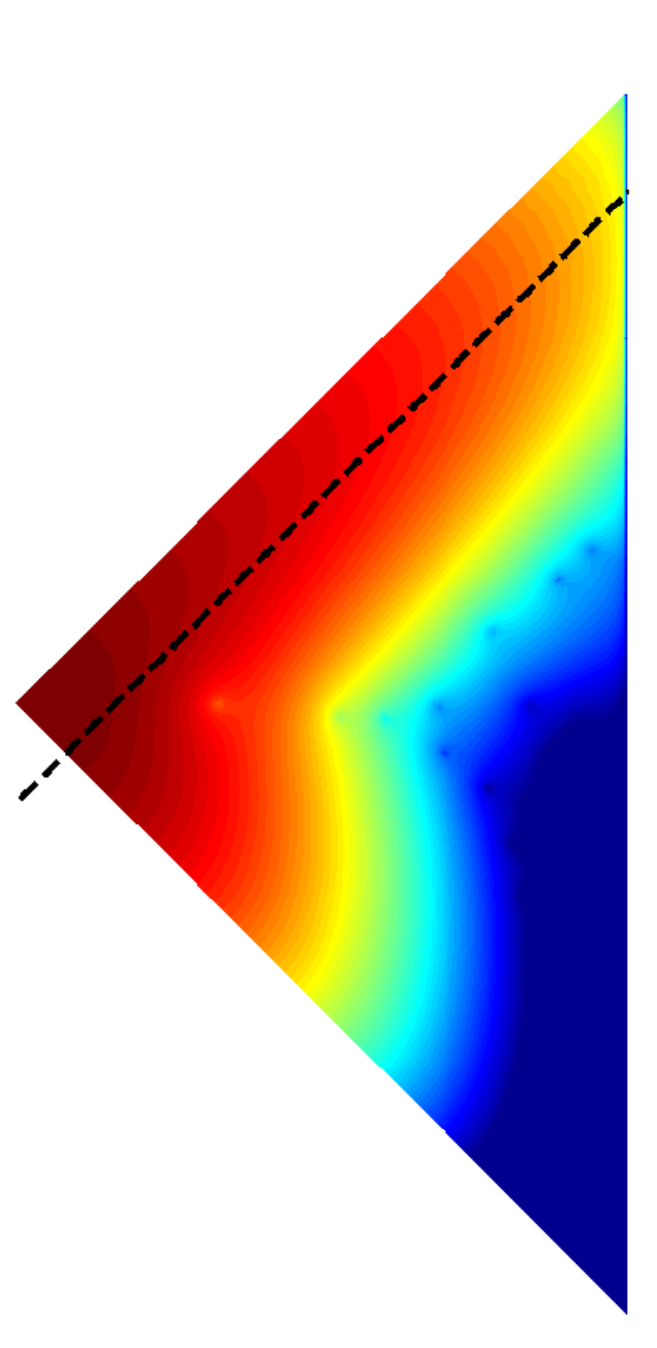}
  &
\includegraphics[keepaspectratio=true,width=2.0cm]{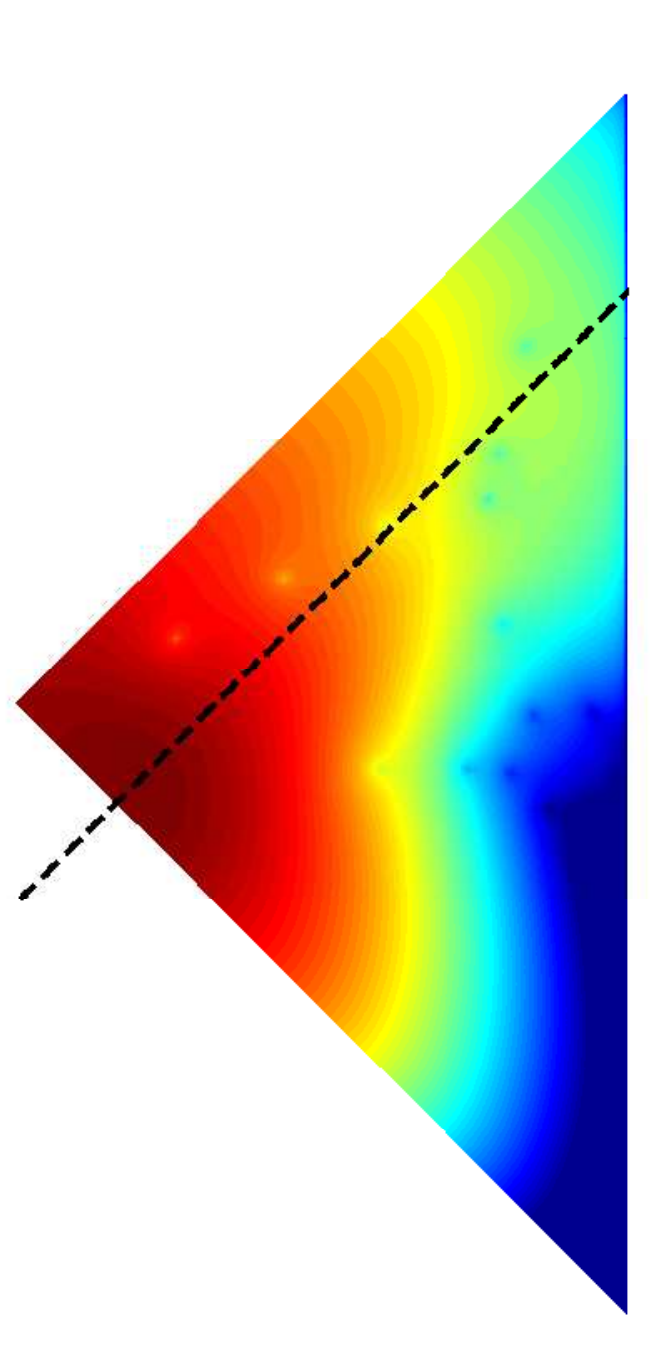}  
&
\includegraphics[keepaspectratio=true,width=2.0cm]{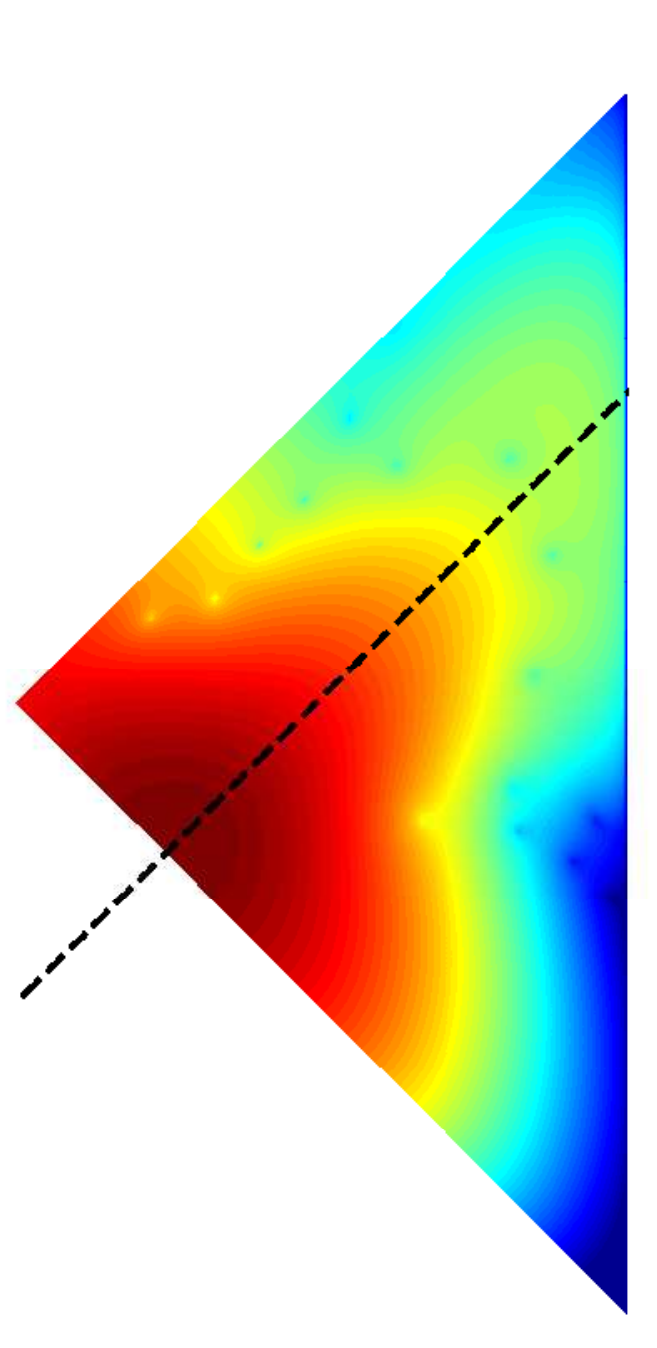}  
&
\includegraphics[keepaspectratio=true,width=2.0cm]{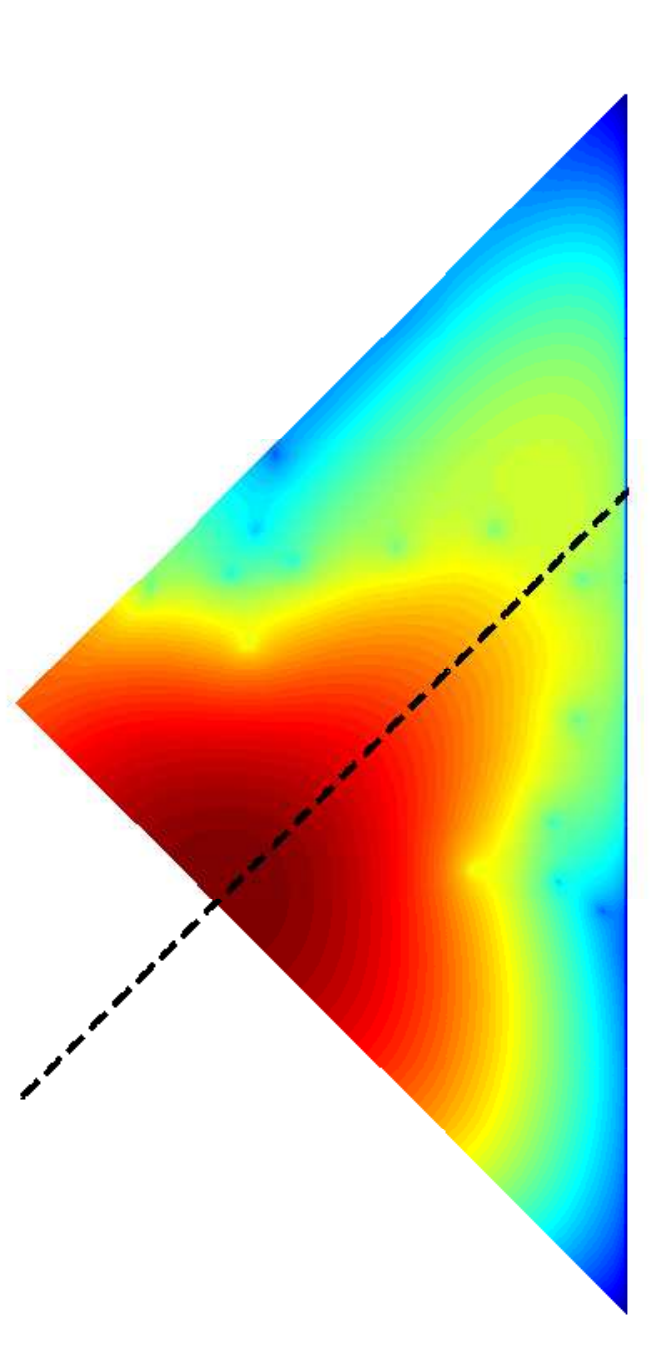} 
&
\includegraphics[keepaspectratio=true,width=0.72cm]{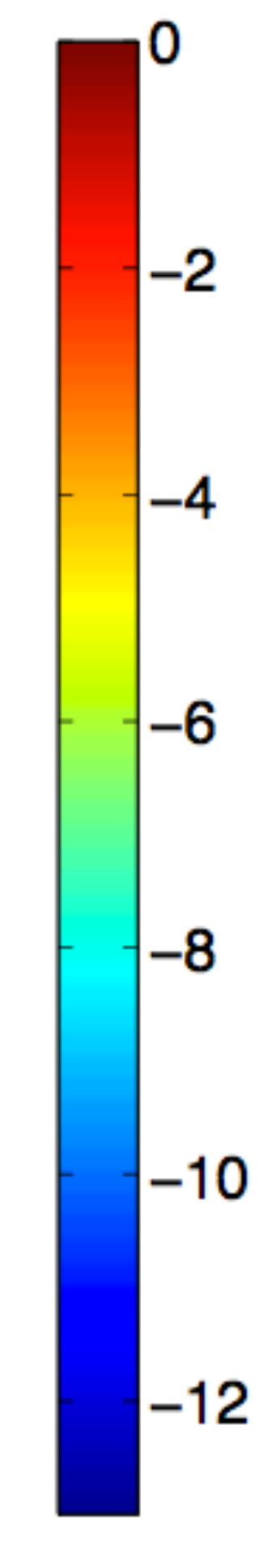}
\\
\includegraphics[keepaspectratio=true,width=2.0cm]{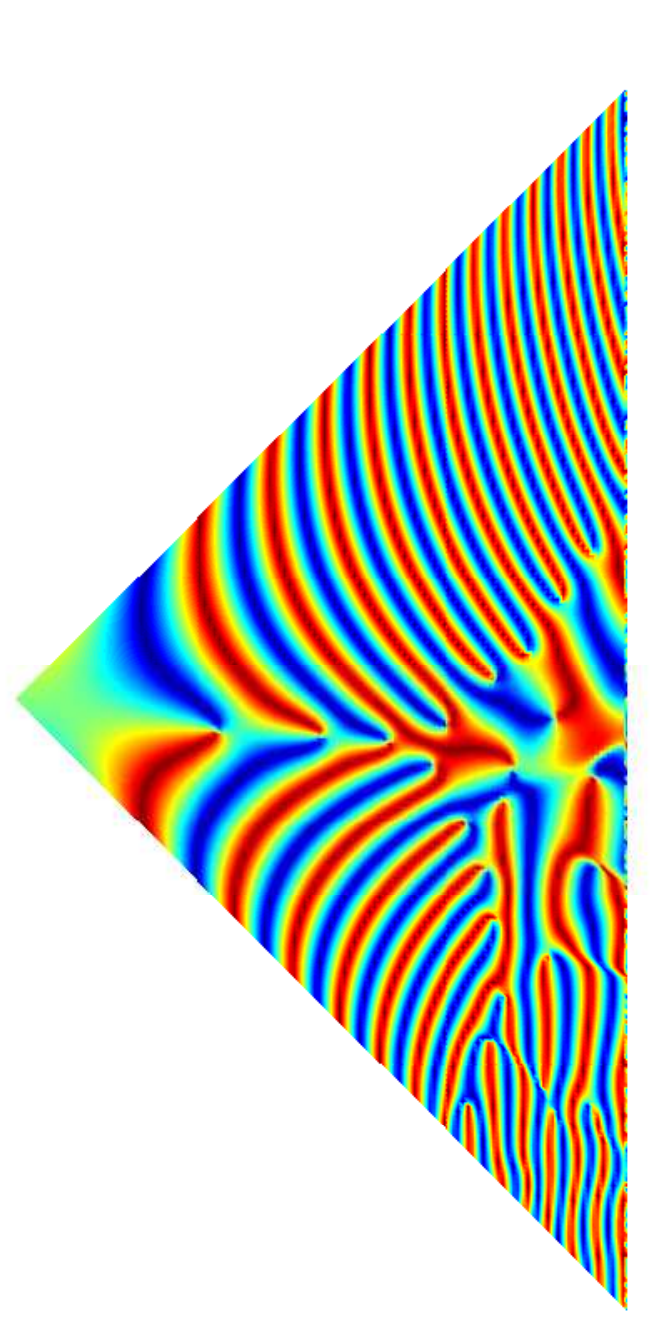}
&
\includegraphics[keepaspectratio=true,width=2.0cm]{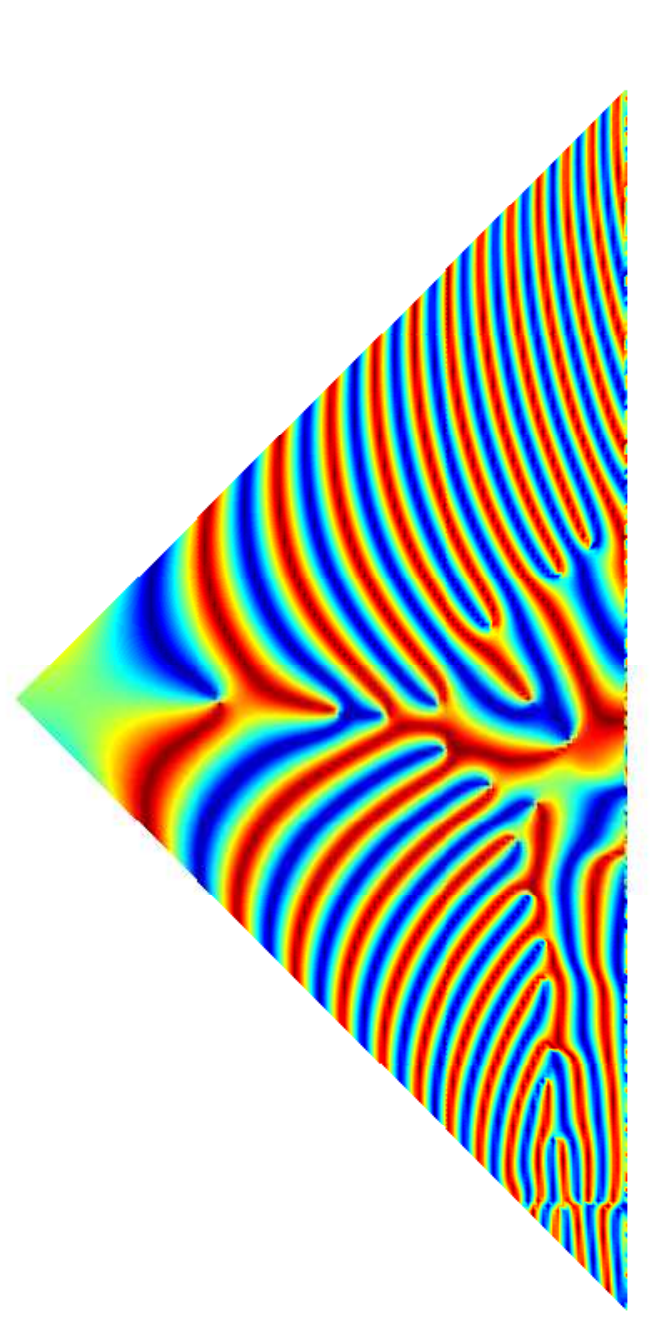}
  &
\includegraphics[keepaspectratio=true,width=2.0cm]{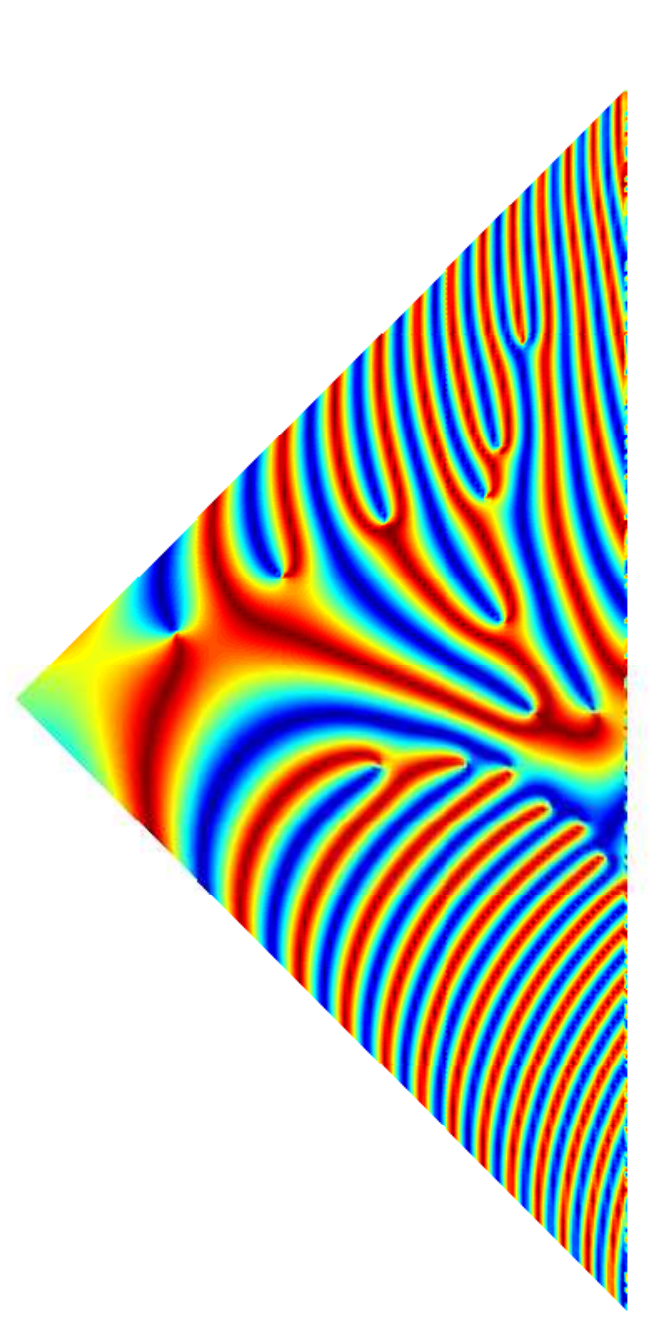}  
&
\includegraphics[keepaspectratio=true,width=2.0cm]{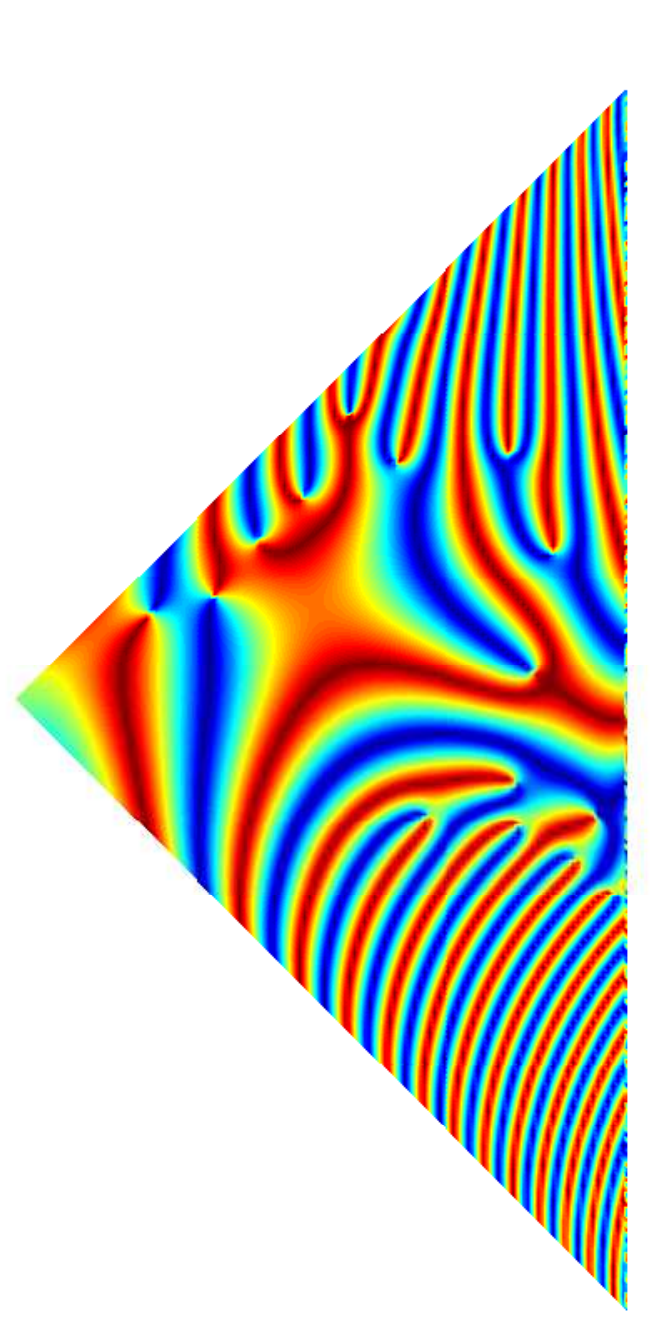}  
&
\includegraphics[keepaspectratio=true,width=2.0cm]{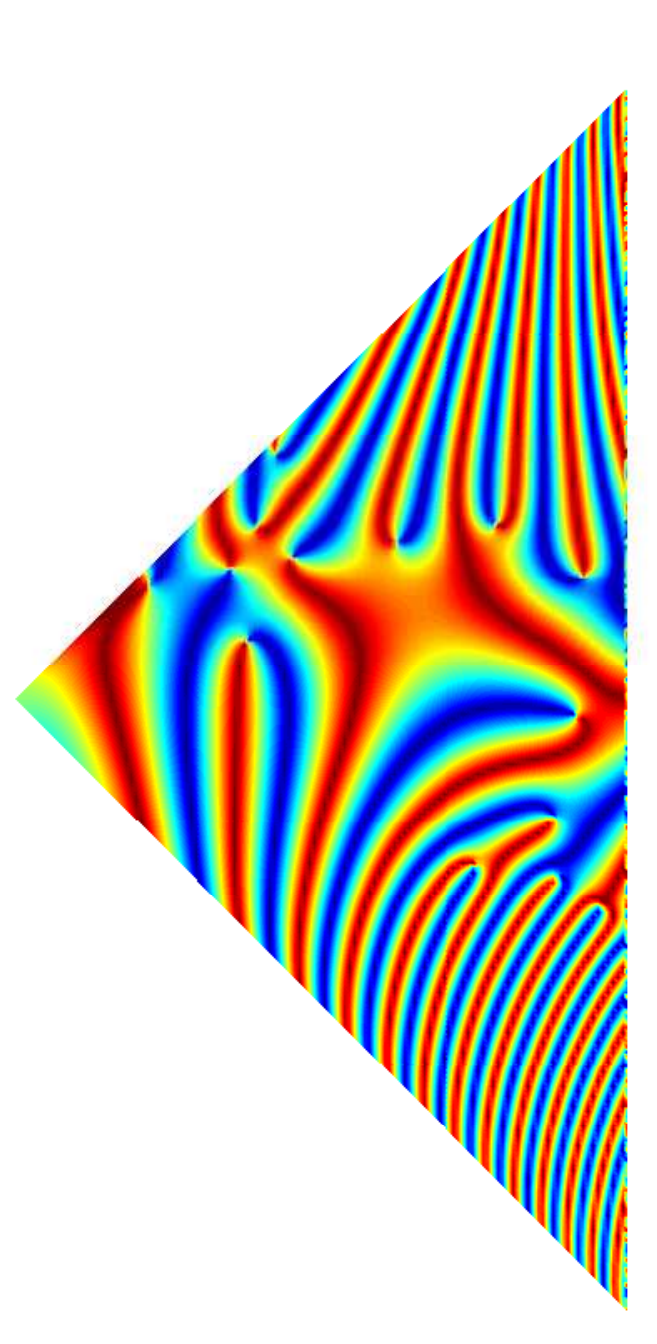} 
&
\includegraphics[keepaspectratio=true,width=0.60cm]{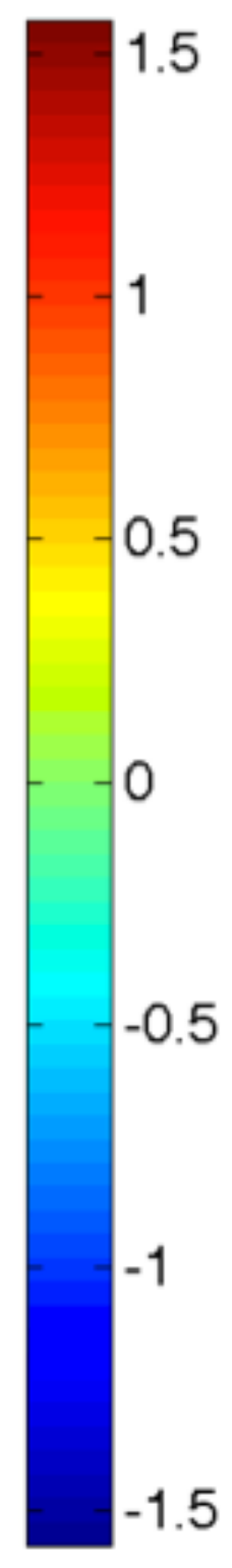}
\\
$\tau=0$ & $\tau=0.5$ & $\tau=1$ & $\tau=1.5$ & $\tau=2$
\end{tabular}
\caption{Opening angle: $\alpha=\frac{\pi}{2}$. Spherical coordinates of ${\bf B}$: $(\gamma,\theta)=(\frac{\pi}{10},\frac{\pi}{4})$. From top to bottom: the modulus, the base-10 logarithm of the modulus and the phases modulo $\pi$ of the eigenfunction associated to $\breve{\sse}({\bf B};\alpha,\tau)$ for $\tau=\frac{k}{2}$, $0\leq k \leq 4$. In dash line : the set $\Upsilon$. Computational domain: $\mathcal{T}_{\frac{\pi}{2},14}$. Magnetic potential: $\Add^{\rm R}$. 
}
\label{F2}
\end{figure}

\newpage
On figure \ref{F4} we take a magnetic field of spherical coordinates $(\gamma,\theta)=(\frac{\pi}{2},\frac{\pi-\alpha}{2})$. The magnetic field is tangent to a face and normal to the edge. For each value of $\alpha$ we make several computations of $\breve{\sse}({\bf B};\alpha,\tau)$ and we define 
$$ \breve{\spectre}({\bf B};\Wedge_{\alpha}):=\inf_{\tau}\breve{\sse}({\bf B};\alpha,\tau)$$ a numerical approximation of $\spectre({\bf B};\Wedge_{\alpha})$. We have plotted $\breve{\spectre}({\bf B};\Wedge_{\alpha})$ for $\alpha=k\frac{\pi}{20}$ with $1 \leq k \leq 19$. We have also plotted the constant $\Theta_{0}$ and the upper bound from Proposition \ref{E:lowerbound}.

We observe that the $\alpha\mapsto\breve{\spectre}({\bf B};\Wedge_{\alpha})$ is non decreasing on $(0,\frac{\pi}{2}]$ and close to $\Theta_{0}$ for $\alpha\in[\frac{\pi}{2},\pi]$, according to Proposition \ref{P:monotonie} and Theorem \ref{T:egalite}. Moreover $\breve{\spectre}({\bf B};\Wedge_{\alpha})$ seems to go to 0 when $\alpha$ goes to 0, according with the results from Section \ref{S:ub}.

\begin{figure}[ht]
\begin{center}
\includegraphics[keepaspectratio=true,width=16cm]{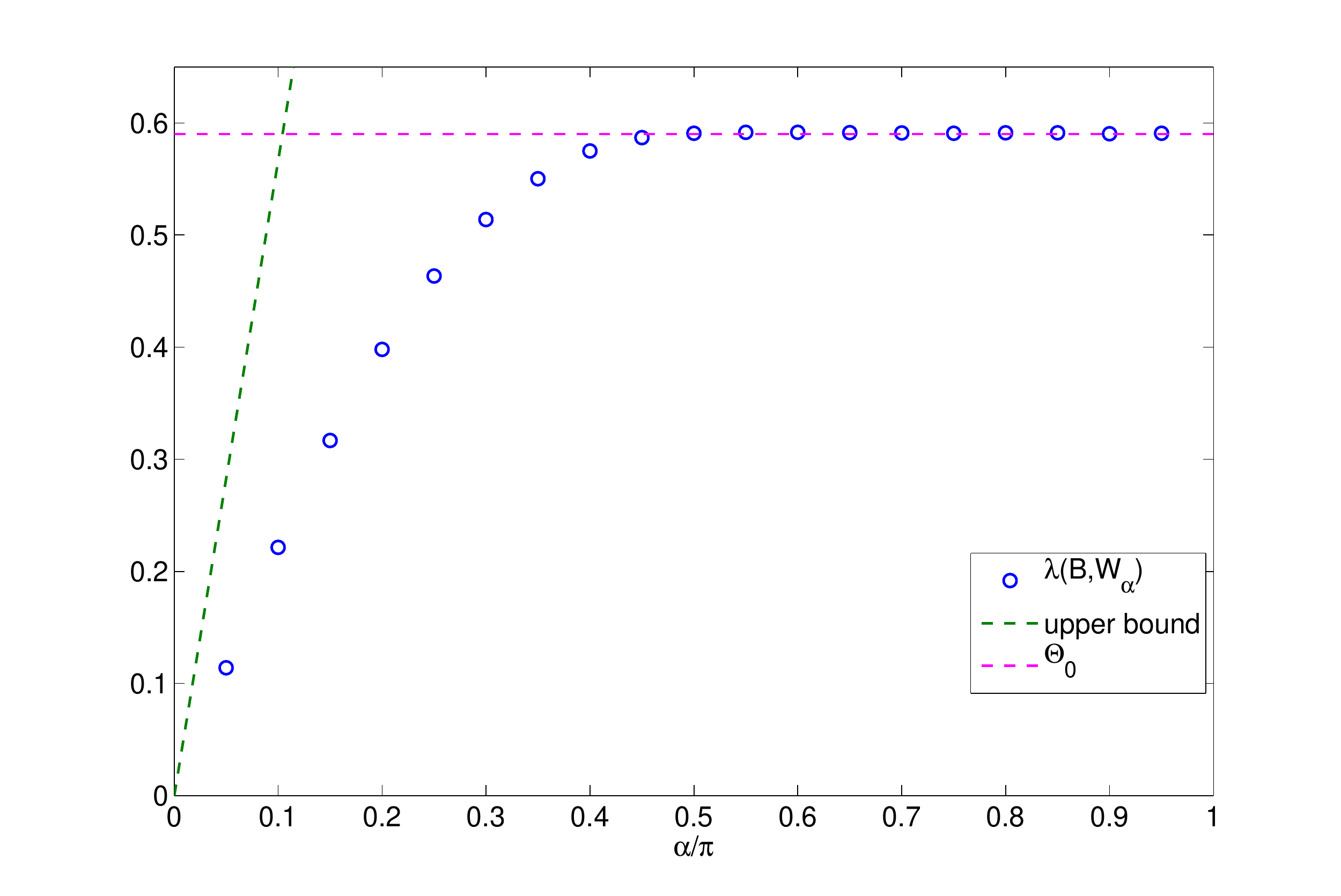}
\caption{Spherical coordinates of ${\bf B}$: $(\gamma,\theta)=(\frac{\pi}{2},\frac{\pi-\alpha}{2})$. The approximation $\breve{\sd}({\bf B};\Wedge_{\alpha})$ with respect to $\vartheta:=\frac{\alpha}{\pi}$ for $\vartheta=\frac{k}{20}$, $1\leq k \leq 19$ compared to $\Theta_{0}$ and to the upper bound from Proposition \ref{E:lowerbound}.}
\label{F4}
\end{center}
\end{figure}
\paragraph{Acknowledgements} The author would like to thank the Mittag-Leffler Institute where this articles has partially been written. The author is also grateful to V. Bonnaillie-No\"el and M. Dauge for their advices and their interest in this work.
\newpage
\bibliographystyle{mnachrn}

\end{document}